\documentclass[a4paper,12pt]{article}
\usepackage[utf8]{inputenc}
\usepackage{fullpage}

\usepackage{mathtools}

\usepackage{amsmath, amssymb, amstext, amsfonts, amsthm, bbm, array, enumerate,scalerel,dsfont}
\usepackage{mathrsfs}%\mathscr schoen, \mathcal.. nicht so schoen
\usepackage{nicefrac}
\usepackage[none]{hyphenat}
\usepackage{xcolor}
\usepackage{tikz}
\usepackage{stmaryrd}
\usepackage{tikz-3dplot}
\usetikzlibrary{patterns}
\usetikzlibrary{plotmarks}

\usepackage{caption}
\usepackage{subcaption}

\newtheorem{theorem}{Theorem}[section]
\newtheorem{lemma}[theorem]{Lemma}

\theoremstyle{definition}
\newtheorem{definition}[theorem]{Definition}

\theoremstyle{remark}
\newtheorem{remark}[theorem]{Remark}

\newcommand{\pp}{\mathbb{P}}
\newcommand{\E}{\mathbb{E}}

\renewcommand{\P}[2][]{\mathbb{P}_{#1}\left(#2\right )}
\newcommand{\R}{\mathbb{R}}

\newcommand{\N}{\mathbb{N}}
\newcommand{\ind}[1]{\mathds{1}_{#1}}
\renewcommand{\d}{\mathrm{d}}

\newcommand{\T}{\mathbb{T}}

\renewcommand{\L}{\mathcal{L}}
\newcommand{\cadlag}{c\`adl\`ag~}

\def\shalf{\text{$\frac{1}{2}$}}

\def\sR{{\mathbb R}}
\def\cH{\mathcal{H}}

\def\sR{{\mathbb R}}
\def\cH{\mathcal{H}}

\def\bb{\begin{equation}} \def\ee{\end{equation}}
\def\bbn{\begin{equation*}} \def\een{\end{equation*}}
\newcommand{\PP}{{\mathbb P}}
 
\DeclareMathOperator*{\prox}{prox}

\newcommand{\ExtendedE}{\bar{E}}
\newcommand{\law}{\operatorname{law}}

\newcommand{\lawXnull}{\nu_{0-}}
\newcommand{\Lmin}{\underline{\Lambda}}
\renewcommand{\L}{\Lambda}
\newcommand{\ellfunc}{\lambda}
\newcommand{\EmpM}{\mu}

\newcommand{\EmpMext}{\xi}
\newcommand{\embedding}{\hat \iota}
\newcommand{\NplayerGame}{controlled $N$-particle system}
\newcommand{\Lpertmin}{\underline{\tilde{L}}{}}
\newcommand{\LpertminLambda}{\underline{\tilde{\L}}{}}
\newcommand{\betaNstar}{(\beta^\star)^N}

\newcommand\mydots{\hbox to 1em{.\hss.\hss.}}
\title{Optimal bailout strategies resulting from  the drift controlled supercooled Stefan problem}

\author{Christa Cuchiero\thanks{Vienna University, Department of Statistics and Operations Research, Data Science @ Uni Vienna, Kolingasse 14-16, A-1090 Wien, Austria, christa.cuchiero@univie.ac.at}
\and Christoph Reisinger \thanks{Mathematical Institute and Oxford Man Institute of Quantitative Finance, University of Oxford, Andrew Wiles Building, Radcliffe Observatory Quarter, OX2 6GG, Oxford, U.K., christoph.reisinger@maths.ox.ac.uk}
\and Stefan Rigger \thanks{Vienna University, Department of Statistics and Operations Research, Kolingasse 14-16, A-1090 Wien, Austria, stefan.rigger@univie.ac.at}}
\date{}

\begin{document}
\maketitle

\begin{abstract}
We consider the problem faced by a central bank which bails out distressed financial institutions that pose systemic risk to the banking sector.
In a structural default model with mutual obligations, the central agent seeks to inject a minimum amount of cash 
 in order to limit defaults to a given proportion of entities.
We prove that the value of the central agent's control problem converges as the number of defaultable institutions goes to infinity,
and that it satisfies  a drift controlled version of the supercooled Stefan problem.
We compute optimal strategies in feedback form by solving numerically 
a regularized version of the corresponding  mean field control problem using a policy gradient method.
%a forward-backward coupled system of PDEs.
Our simulations show that the central agent's optimal strategy is to subsidise banks whose equity values lie in a non-trivial time-dependent region.
%Finally, we study a linear-quadratic version of the model where instead of the terminal losses, the agent optimises a terminal cost function of the equity values.
%In this case, we are able to give semi-analytic strategies, which we again illustrate numerically.
\end{abstract}

Propagation of chaos, Bail-outs

\section{Introduction}

In this paper, we analyse a simple mathematical model for a central bank that optimally injects cash into a banking system with interbank lending in order to
prevent systemic default events.
By way of introduction, we first review known results on the dynamics without intervention and its relation to the supercooled Stefan problem.
We then describe the optimisation problem faced by the central agent and discuss its setting within the literature on Mean Field Control (MFC) problems together with this paper's
contributions.

\subsection{Interbank lending and the supercooled Stefan problem}\label{sec:1.1}

We study a market with $N$ financial institutions and their equity value process $X=(X_t^i)$ for $t\in [0,T]$ with finite time horizon $T >0$ and $i=1,\ldots, N$. We interpret $X^i$  in the spirit of structural credit risk models as the value of assets minus liabilities. Hence, we consider an institution to be defaulted if its equity value hits 0.
We refer the reader to \cite{merton1974pricing} for the classical treatment as well as to \cite{itkin2017structural} and the references therein for a discussion of such models in the
present multivariate context with mutual obligations.

We consider specifically a stylised model of interbank lending where all
firms are ex\-change\-able, their equity values are driven by Brownian motion,
and where
the default of one firm leads to a uniform downward jump in the equity value of the surviving entities.
The latter effect is the crucial mechanism for credit contagion in our model as it describes how the default of one firm affects the balance sheet of others.
Here, we follow \cite{HLS,NS1,lipton2019semi}) to assume that the $X^i$ satisfy 
\begin{align}\label{eq::uncontrolled}
	X^i_t &= X^i_{0-} %+ \int_0^t \beta^i_s \, ds 
	+ B^i_t - \alpha \frac{1}{N} \sum_{i=1}^N \mathds{1}_{\{\tau^i \le t\}},\;\;
%	\Lambda_t &= \alpha\pp\Big(\inf_{0\le s\le t} X_s \le 0\Big).\;\;
\end{align}
where $\tau^i = \inf\{t: X^i_t \le 0 \}$, $X^i_{0-}$ are non-negative i.i.d.\ random variables, $(B^i)_{1\le i\le N}$ is an $N$-dimensional standard Brownian motion, independent of
$X_{0-}=(X^i_{0-})_{1\le i\le N}$, 
and $\alpha \ge 0$ is a given parameter measuring the 
interconnectedness in the banking system. 
The initial condition reflects the current state of the banking system. This might include minimal capital requirements prescribed by the regulator as conditions to enter the banking system, but we do not consider this question explicitly.

Even this highly stylised  simple system produces complex behaviour for large pools of firms, including systemic events where cascades of defaults caused by interbank lending instantaneously wipe out significant proportions of the firm pool (see \cite{HLS,NS1,DNS}).

One way of analysing this is to pass to the mean-field limit for $N \to \infty$. It is known (see, e.g., \cite{DIRT}) 
that the interaction (contagion) term in \eqref{eq::uncontrolled} converges  (in an appropriate sense) to a deterministic function $\Lambda: [0,T] \rightarrow [0,\alpha]$ as  $N\rightarrow \infty$, i.e.,
\[
\alpha \frac{1}{N} \sum_{i=1}^N \mathds{1}_{\{\tau^i \le t\}} \; \rightarrow \; \Lambda_t. %\;\; t\ge 0.
\]
Moreover, the $X^i$ are asymptotically independent with the same law as a process $X$ which together with $\Lambda$
satisfies a probabilisitic version of the \emph{supercooled Stefan problem}, namely 
\begin{equation}\label{eq::init1}
	X_t = X_{0-} + B_t - \Lambda_t,\;\; t\ge 0, 
\end{equation}
where $\Lambda$ is subject to the constraint
\begin{equation}\label{eq::init2}
	\Lambda_t = \alpha\pp\Big(\inf_{0\le s\le t} X_s \le 0\Big),\;\;t\ge 0.
\end{equation}
Here, $B$ is a standard Brownian motion independent of the random variable $X_{0-}$, which has the same law as all $X^i_{0-}$. We refer to \cite{DNS}
for a discussion on
how this probabilistic formulation relates to the classical PDE version of the supercooled Stefan problem.

%We assume here that $X_{0-}$ has a probability density $f$ supported on $[0,\infty)$, i.e.\ $\|f\|_{L^1([0,\infty))}=1$.

From a large pool perspective (see \cite{HLS, NS1}), $X_t$ may be viewed as the equity value of a representative bank and
$\tau = \inf \{t\ge 0: X_t \le 0\}$
as its default time, while  $\Lambda_t$ describes the interaction with other institutions under 
the assumption of uniform lending and exchangeable dynamics.
In particular, 
$\pp(\inf_{0\le s\le t} X_s \le 0)$
can be interpreted as the fraction of defaulted banks at time $t$ and consequently $\Lambda_t$ as the loss that the default
of these entities has caused for the survivors.

%For future reference, we define the stopped process $\widehat X_t = X_t \mathds{1}_{\tau> t}$.

It is known that solutions to \eqref{eq::init1}, \eqref{eq::init2} are not unique in general (see \cite{DIRT,CRS}), which explains the need to single out so-called \emph{physical} solutions that are meaningful from an economic and physical perspective.  Under appropriate conditions  on $X_{0-}$, these physical solutions are characterised by open intervals with smooth $t \mapsto \Lambda_t$, separated by
points at which this dependence may only be H{\"o}lder continuous or even exhibit a discontinuity, an event frequently referred to as \emph{blow-up}
(see \cite{DNS}). If the mean of the initial values is close enough to zero relative to the interaction parameter $\alpha$, a jump necessarily happens (see \cite{HLS}).

In case a discontinuity does occur at some $t\ge 0$, the following restriction on the jump size defines such a \emph{physical solution}:
\begin{equation}\label{physical}
	\Lambda_t-\Lambda_{t-}=\inf\Big\{x>0:\,\pp\big(\tau\ge t,\,X_{t-}\in(0,x]\big)<\frac{x}{\alpha}\Big\},\;\;t\ge 0, 
\end{equation} 
with $\Lambda_{t-}:=\lim_{s\uparrow t} \Lambda_s$ and $X_{t-}:=\lim_{s\uparrow t} X_s$.
By the results of \cite{DNS}, the above condition on the jumps of $\Lambda$ uniquely defines a solution to \eqref{eq::init1} and \eqref{eq::init2} under some restrictions on the initial condition $X_{0-}$. 
For future reference, we also introduce the concept of \emph{minimal solutions}, which we know to be physical whenever the initial condition is integrable
(see \cite{CRS}).
We call a solution $\underline{\Lambda}$ \emph{minimal}, if for any other $X$ that satisfies \eqref{eq::init1} with 
loss process $\Lambda$ given by \eqref{eq::init2}, we have that 
\begin{equation}\label{eq:defminimal}
\underline{\Lambda}_t \leq \Lambda_t, \quad t \geq 0.
\end{equation}
Note that by combining the results of \cite{CRS} and \cite{DNS} the minimal solution is the unique physical one, and thus ecnomoically meaningful one, 
if the initial condition satisfies the assumptions of \cite{DNS}.

\subsection{The central agent's optimisation problem}

The purpose of this paper is to analyse strategies that a central bank (\emph{central agent}) can take to limit the number of defaults. They achieve this by controlling the rate of capital injected to distressed institutions. 
That is to say, rather than bailing out firms which are already defaulted, the central agent intervenes already ahead of the time their equity values become critical. This rate of capital\footnote{
As we are working in continuous time, we also assume that the central agent is able to inject money continuously.
At first sight this is a bold approximation of reality but allows us to work without an a-priori fixed time grid determining when the central agent reacts.
%We sketch possible extensions in Section \ref{sec:concl}.
}
received by bank $i$ is determined by processes $\beta^i$ and added to \eqref{eq::uncontrolled}.
In the finite dimensional situation, the $X^i$ then satisfy 
\begin{align}\label{eq::controlledparticle}
	X^i_t &= X^i_{0-} + \int_0^t \beta^i_s \, ds 
	+ B^i_t - \alpha \frac{1}{N} \sum_{i=1}^N \mathds{1}_{\{\tau^i \le t\}}. \;\;
%	\Lambda_t &= \alpha\pp\Big(\inf_{0\le s\le t} X_s \le 0\Big).\;\;
\end{align}

A mathematically similar problem has been studied in \cite{TT:18}. There the question of finding an optimal drift in order to maximize the
number of Brownian particles that stay above $0$ is treated, however without the singular interaction term appearing in \eqref{eq::controlledparticle}.

In anticipation of a propagation of chaos result (proved in Section \ref{sec:convergence}), we therefore consider an extension of \eqref{eq::init1} and \eqref{eq::init2} with a drift process $\beta$, i.e.,
\begin{align}\label{eq::controlled}
	X_t &= X_{0-} + \int_0^t \beta_s \, ds + B_t - \Lambda_t,\;\; \\
	\label{eq::controlled2}
	\Lambda_t &= \alpha\pp\Big(\inf_{0\le s\le t} X_s \le 0\Big).\;\;
\end{align}
%For most of
Throughout the paper we will consider a constraint $0 \leq \beta_t \leq b_{\text{max}}$, which amounts to the assumption that at any point in time the central agent has limited resources for the capital injections. We will specify further technical conditions on $\beta$ later, which allow us to show that indeed the finite system converges
in a suitable sense to this McKean--Vlasov equation. %mean-field problem.
%We will then analyse the resulting singular %mean field control 
%MFC problem and provide numerical illustrations of the control strategies.

%Specifically, w
We now consider a central agent who injects capital into a representative bank at rate $\beta_t$ at time $t$ in order 
to keep
\[
L_{T-}(\beta) = \pp\Big(\inf_{0\le s< T} X_s \le 0\Big) = \Lambda_{T-}(\beta)/\alpha,
\]
that is the number of defaults
that occur before\footnote{We consider $L_{T-}$ rather than $L_{T}$ in the constraint because $\Lambda$ may have a jump discontinuity precisely at $T$,
which would considerably complicate the analysis. However, by \cite[Corollary 2.3]{LS}, we know that solutions to \eqref{eq::uncontrolled} cannot have discontinuities
after time $\alpha^2/(2\pi)$. We can derive an analogous result for the controlled case using Girsanov's theorem if, in addition to the pointwise bound on $\beta$, we assume a bound on the total cost over the infinite horizon, i.e.\ $\int_{0}^{\infty} \beta_s~\mathrm{d}s \leq c_{\text{max}}$ for some $c_{\text{max}} > 0$. In this setting, by choosing $T$ sufficiently large, we then do not need to distinguish between $L_{T-}(\beta)$ and $L_{T}(\beta)$.}
a given time $T$,
below a specified threshold $\delta$, while
minimising the expected total cost
\[
C_T(\beta) = \E \Big[ \int_0^T \beta_t \, dt \Big].
\]
We therefore consider the following constrained optimisation problem: For given $\delta$, the central agent solves 
\begin{equation}\label{constr_opt}
C_T(\beta) \longrightarrow \min_{\beta} \qquad \text{subject to} \qquad L_{T-}(\beta) \le \delta.
\end{equation}
Define now for $\gamma \in \mathbb{R}_+$ the Lagrange function 
$
\mathcal{L}(\beta, \gamma) = C_T(\beta)+ \gamma (L_{T-}(\beta) - \delta)
$
and use it to express the constrained optimization problem as an unconstrained one, namely $ \min_{\beta} \max_{\gamma \in \mathbb{R}_+} \mathcal{L}(\beta, \gamma)$, which holds true since
\[
 \max_{\gamma \in \mathbb{R}_+} \mathcal{L}(\beta, \gamma)= \begin{cases} C_T(\beta) & \text{if } L_{T-} \leq \delta\\
\infty  & \text{else.}
 \end{cases}
\]
Assuming the absence a duality gap\footnote{Proving the  absence a duality gap  seems difficult as standard minimax theorems cannot be easily applied. Moreover, our numerical experiments suggest that at least for certain values of $\delta$ it might fail to hold true.} (or equivalently the existence of a saddlepoint $(\beta^\star, \gamma^*)$ of  $\mathcal{L}$, i.e.
$
\mathcal{L}(\beta^\star, \gamma) \leq \mathcal{L}(\beta^\star, \gamma^*) \leq \mathcal{L}(\beta, \gamma^*) 
$
for all $\beta, \gamma$), 
then  we can interchange the $\min$ and $\max$ and solve the dual problem  $ \max_{\gamma \in \mathbb{R}_+} \min_{\beta} \mathcal{L}(\beta, \gamma)$.

%and suppose that a saddlepoint $(\beta^\star, \gamma)$ of  $\mathcal{L}$ exists, i.e.
%$
%\mathcal{L}(\beta^\star, \lambda) \leq \mathcal{L}(\beta^\star, \gamma) \leq \mathcal{L}(\beta, \gamma) %\quad \quad \forall \beta, \lambda.
%$
%for all $\beta, \lambda$.
%Then 
%$
%\mathcal{L}(\beta^\star, \gamma)=C_T(\beta^\star)= 
%\inf\{C_T(\beta) \, : \,  L_{T-}(\beta) \le \delta \}
%$
%and instead of solving \eqref{constr_opt}, the central agent can equivalently solve 
%$
%\min_{\beta}\mathcal{L}(\beta, \gamma) %\longrightarrow 
%$
%for the optimal Lagrange multiplier $\gamma$.
For these reasons we shall from now on consider the inner optimisation problem for fixed $\gamma >0$ (which can -- due to the complementary slackness condition -- only hold if   the  constraint is binding, i.e.~$L_{T-}(\beta)=\delta$). If there is no duality gap, the optimal $\gamma$ for a prespecified threshold $\delta$ can in turn be determined by solving the outer optimisation problem, i.e.~$\max_{\gamma \in \mathbb{R}_+} g(\gamma)$ where $g(\gamma)= \min_{\beta} \mathcal{L}(\beta, \gamma)$.

Writing $\underline{X}(\beta)$ for the solution process associated with the minimal solution $\underline{\Lambda}(\beta)$, analogously defined as in \eqref{eq:defminimal} but now for \eqref{eq::controlled},
we thus minimise the following objective function
\begin{eqnarray}
\nonumber
	J(\beta) &=& \E \Big[ \int_0^T \beta_t \, dt \Big] + \gamma \, \pp\Big(\inf_{0\le s < T} \underline{X}_s(\beta) \le 0\Big) \\
%\nonumber	&=& \E \Big[ \int_0^T \beta_t \, dt + \gamma \,  \mathds{1}_{\{\underline{\tau}_{\beta} < T\}} \Big] \\
	&=& \E \Big[ \int_0^T \beta_t \, dt + \gamma \,  \mathds{1}_{\{\widehat{\underline{X}}_{T-} = 0\}} \Big], 
	\label{eq::objective}
\end{eqnarray}
where $\widehat{\underline{X}}=\underline{X}_t \mathds{1}_{\{\tau> t\}}$ is the absorbed minimal solution process and $\tau$ the default time. Note that the only difference between $J(\beta)$ and $\mathcal{L}(\beta, \gamma)$ is the constant $-\gamma \delta$, which however does not play a role in the optimisation over $\beta$. 
By varying $\gamma$, we can therefore trace out
pairs of costs and losses which are solutions to \eqref{constr_opt} for different $\delta$.
The  Lagrange multiplier $\gamma$ (as a function of $\delta$) can then be interpreted as shadow price of preventing further defaults. Indeed, as for usual constrained optimization problems,  the optimal cost $C^{\star}_T$ seen as a function of the loss level $\delta$ 
satisfies under certain technical conditions
\[
\partial_{\delta} C_T^{\star}(\delta)=\lim_{h \to  0}\frac{C^\star_T(\delta +h)-C^\star_T(\delta )}{h}= - \gamma(\delta).
\]
As we show numerically in  Section \ref{sec:3a}, the optimal loss $L_{T-}^\star$ as a function of $\gamma$  is monotone decreasing, so that for  large enough $\gamma$ (and $b_{\max}$), the threshold $\delta$ becomes small enough to avoid systemic events.
%\fbox{verify that this is still true}

Note that, by the arguments at the end of Section \ref{sec:1.1}, using the minimal solution in the optimisation task is the only economically meaningful concept because non-physical solutions (with a potential higher
probability of default) cannot be realistically justified, in particular when seen as limits of  particle systems. We refer to \cite[Section 3.1]{DIRT} for examples of such non-physical solutions.

Both from a theoretical and numerical perspective, 
we shall %now further 
analyse 
the  objective function \eqref{eq::objective} together with the dynamics \eqref{eq::controlled}, 
which
is  a non-standard \emph{MFC problem with a singular interaction} through hitting the boundary.
%, and its relation to existing literature. 
As we show in  Section \ref{sec:convergence}, in particular Theorem \ref{thm::optimizersconverge}, optimisation of the McKean--Vlasov equation \eqref{eq::controlled} yields the same result as first optimising in the $N$-particle system and then passing to the limit. In particular, by Theorem \ref{thm:epsilonoptimality}, optimizers of the McKean--Vlasov equation \eqref{eq::controlled} are $\epsilon$-optimal for the $N$-particle system.
This then justifies our numerical implementation described in Section \ref{sec:3a} where we deal directly with the MFC problem.
%In this sense, the mean-field and McKean--Vlasov control problems coincide
%in this setting. %lead to the same result. 
%We refer to \cite{CDL:13} for a discussion on the differences in general set-ups.

\subsection{Relation to the literature}

\subsubsection*{Theory of MFC problems and applications to systemic risk}

Due to the big amount of literature on MFC problems 
%mean-field control and McKean--Vlasov control problems 
we focus here on relatively recent  works %  and the references therein.
and mostly on MFC and not on the related concept of Mean Field Games (MFG) as introduced in \cite{lasry2007mean} and \cite{huang2006large}. 
We refer to \cite{CDL:13, carmona2015forward, CL:21} for definitions of the MFC and MFG optima in  general set-ups and discussions on the differences. As we here deal with a central agent our optimization problem corresponds to a Pareto optimum  where all the agents cooperate to minimize the 
costs. Therefore MFC is the appropriate concept.
Note that instead of MFC the terminology
\emph{McKean--Vlasov control} 
is often also used.

Similarly as for classical optimal control, dynamic programing principles have also been derived for 
MFC problems and can be found in \cite{PW:17, DPT:19}. We also refer to  
\cite{pham2018bellman, BFY:15, achdou2015system,lauriere2014dynamic}, where in diffusion set-ups formulations using a Hamilton-Jacobi Bellman (HJB) equation on the space of probability measures for closed-loop controls (also called feedback controls) are deduced. In the recent work \cite{GPW:20} this has been generalized to jump diffusion processes.
  For a dynamic programming principle for open-loop controls we refer to  \cite{bayraktar2018randomized},  and  to \cite{carmona2015forward, acciaio2019extended} for  a characterisation by a stochastic maximum principle.

 % In \cite{achdou2015system} a characterisation as a coupled system of nonlinear forward and backward PDEs is provided, following the pioneering work on mean-field games (MFGs) in \cite{lasry2007mean} and \cite{huang2006large} (see also \cite{bensoussan2013mean}).

We are here interested in feedback controls and would therefore need to solve the corresponding HJB equation (as e.g.~in \cite{pham2018bellman}), i.e. an
infinite dimensional fully
nonlinear partial differential equatios (PDE) of second order in the Wassertein space of
probability measures. Solving such an equation is challenging, since it involves  computing measure derivatives, which is numerically  intractable. In our context the situation is even more intricate due to the singular interactions through the boundary. Indeed, even under the (usually not satisfied) assumption that $t \mapsto \Lambda_t$ is $C^1$, the problem is far beyond a standard MFC framework. In this case, 
$\Lambda$ in \eqref{eq::controlled} can be replaced by $\int_0^{\cdot} \dot \Lambda_t dt $, thus a time derivative of
the measure component, which makes the problem ‘non-Markovian’.
Moreover, we deal with subprobability measures 
describing the marginal distributions of the absorbed process 
$\widehat{X}= X_t \mathds{1}_{\{\tau> t\}}$ which governs the underlying dynamics. 
Note also
that the total mass of these subprobability measures as well as $\widehat{X}$ itself can  exhibit jumps if $\Lambda$ is discontinuous,
and that these jumps  emerge endogenously from the feedback mechanism.

This is in contrast to some other recent papers  where jumps are exogenously given. 
For instance, the recent articles \cite{GPW:20, agram2021fokker} consider the control of (conditional) McKean--Vlasov dynamics with jumps and associated HJB-PIDEs,  
while in \cite{ACDZ:21} a stochastic maximum principle is derived to analyse a mean-field game with random jump time penalty.

In the context of  systemic risk and contagion via singular interactions through hitting times the paper  \cite{NS:20} is especially relevant. There a game in which the banks determine their (inhomogeneous) connections strategically is analysed. It turns out that by a reduction of lending to critical institutions in equilibrium systemic events can be avoided.
A model involving singular interaction through hitting the boundary is also considered in  \cite{EIL:20}. 
There, an optimization component  is incorporated via a quadratic functional that allows
the institutions to control
parts of their dynamics in order to minimize their expected risk, which then leads to a MFG problem.
The quadratic cost functional is inspired by the earlier work \cite{CFS15}, which also
treats the mean-field game limit of a system of banks who  control their borrowing from, and lending to, a central bank, and where the interaction comes from interbank lending. Let us finally mention the very recent article \cite{ADFL:22} 
which applies reinforcement learning to a model
that can be considered as an extension of \cite{CFS15} 
adding a cooperative game component within certain groups of banks.

In the wider context of interaction through boundary absorption, a few works on mean-field games have also appeared recently.
In \cite{campi2018n}, the players' dynamics depends on the empirical measure representing players who have not exited a domain.
This is extended to smooth dependence on boundary losses prior to the present time in \cite{campi2019n},
and to the presence of common noise in \cite{burzoni2021mean}.
The economic motivation for these models are, among others, systemic risk and bank runs.

\subsubsection*{Numerics for MFC and MFG problems}

Among the numerical methods proposed for MFC and MFG problems, we refer to \cite{carmona2019convergence} for a policy gradient-type method
where feedback controls are approximated by neural networks and optimised for a given objective function;
to \cite{fouque2020deep} and again to \cite{carmona2019convergence} for a mean-field FBSDE method, generalising the deep BSDE method to mean-field dependence and in the former case to delayed effects;
 and
to \cite{achdou2020mean} for a survey of methods for the coupled PDE systems, mainly
in the spirit of the seminal works \cite{achdou2010mean, achdou2016mean}; see also a related semi-Lagrangian scheme in \cite{carlini2014semi}  and a gradient method and penalisation approach in \cite{pfeiffer2017numerical}.
%and a recent analysis of policy iteration in \cite{lauriere2021policy}.

Beside these works on PDE systems,  a lot of research has recently been conducted on how
to apply (deep) reinforcement or Q-learning to solve MFC and MFG problems or combinations thereof (see e.g., \cite{guo2019learning, angiuli2022unified, angiuli2021reinforcement, angiuli2022reinforcement} and the references therein). We also refer to two recent survey articles \cite{hu2022recent, lauriere2022learning} on machine learning tools for optimal control and games. The first article focuses on methods  that try to solve the problems  by relying on exact computations of gradients exploiting
the full knowledge of the model, while the second one presents learning tools that aim at solving  MFC and MFG problems in a model-free fashion.

In our modeling and numerical approach we have full knowledge of the model, but 
what distinguishes it from the existing literature is the particular singular interaction through the boundary absorption. This means that 
all the discussed numerical schemes and methods need to be adapted to accommodate the current special situation. We opted for an adaptation of the 
policy gradient-type method considered in \cite{reisinger2021fast}, since it shares the same computational complexity as the gradient-based algorithms in e.g. \cite{archibald2020efficient, kerimkulov2021modified, pfeiffer2017numerical}, but enjoys an accelerated convergence rate and can handle general convex nonsmooth costs (including  constraints), which allows to incorporate the current objective function.  It exploits a forward-backward splitting approach and iteratively refines the
approximate controls based on the gradient of the cost, evaluated through a coupled system of nonlocal linear PDEs. The precise algorithm is outlined in Section \ref{sec:3a}.

% It can be seen as hybrid approach in the sense that  the mean-field distribution is
%approximated by a particle system and the control is obtained by numerical approximation of a PDE.

%\bigskip

\subsection{Contributions and findings}

As already mentioned, our model differs in a number of fundamental points from the existing literature:  first, while 
\cite{CFS15, campi2018n,campi2019n,NS:20, EIL:20, burzoni2021mean} study mean-field game solutions and equilibria of $N$-player games, where each player maximises their own objective, we study the problem of a central planner who specifically seeks to control the number of defaults.
Second, in contrast to \cite{campi2019n}, where the coefficients of the players' processes may depend on the loss process
and to \cite{burzoni2021mean}, which further includes a driver which is a smoothed version of $L$ (hence modeling delayed effects of hitting the boundary), we consider
the firm values driven by $L$ directly, resulting in an instantaneous effect of defaults and the emergence of systemic events.
Third, the techniques we use are also entirely different from those in these preceding works. Instead of relying on techniques for martingale problems used e.g. in
\cite{lacker2017limit}  to derive a limit theory for controlled McKean--Vlasov dynamics,  
%\cite{carmona2016mean},
we extend the method from \cite{CRS}
to show the convergence of the finite system to the mean-field limit. 

Moreover, we provide a numerical solution
adapting the approach of \cite{reisinger2021fast} to the current setting. This means first of all to express $\Lambda_t$ in the dynamics \eqref{eq::controlled} and in the loss function explicitly in terms of the distribution of $\widehat{X}_t$, 
which can be (formally) achieved through $\partial_x p(t,0)$, i.e.\ the spatial derivative of its density at $0$ (assuming it exists). 
To cast the absorbed process formally into a (more) standard McKean--Vlasov framework, instead of \eqref{eq::controlled}, we write the dynamics in a form where the drift and diffusion coefficients are multiplied with the Heaviside function (in the state). Both the computation of $\partial_x p(t,0)$ and the presence of the Heaviside function need some regularization, which is treated in Section \ref{sec:reg}. This regularized version then allows to apply the policy gradient descent method of \cite{reisinger2021fast} where we compute the gradient via a 
coupled forward-backward PDE system (instead of a particle method as in \cite{reisinger2021fast}). 
In particular, the forward problem is given by a smoothed version of the Stefan problem with a drift term determined by the feedback control, while the backward problem determines a decoupling field
of the adjoint process.

%\fbox{update this paragraph}
From an economic point of view, our findings indicate a high sensitivity on the parameter $\gamma$ in \eqref{eq::controlled}.  As shown in Figure \ref{fig:C_L_alpha_15},
for a certain value of $\gamma$ and in a regime where $\alpha$ triggers jump discontinuities in the uncontrolled regime, the optimal control strategy switches from 
not avoiding a jump to avoiding a jump.
Moreover, our numerical experiments suggest that it is not possible to vary the capital injection to control the \emph{size} of the jump continuously, since the possible jump size is restricted
by the physical jump constraint  
\eqref{physical}. Viewed differently, a large systemic event can happen if the central agent withdraws a small amount of capital from a scenario without jumps. 

Summarizing, the main contributions of the present paper are as follows:
\begin{itemize}
\item
We show  convergence of the system with $N$ agents to the mean-field limit (see Section \ref{sec:convergence}),
including well-posedness of the central agent's optimisation problem, i.e.\ the existence of unique minimal solutions to the Stefan problem as given by \eqref{eq::controlled}-\eqref{eq::controlled2} for any suitably regular control process $\beta$ and the existence of an optimal control which minimises \eqref{eq::objective}.

\item
%We provide
%a formal derivation of the coupled optimality system of forward and backward PDEs satisfied by the density of the equity value process and the central agent's value function, respectively, and 
We propose a numerical scheme (see Section~\ref{sec:3a}) based on policy gradient iteration, where the gradient is computed via coupled forward and backward PDEs satisfied by the density of a regularized version of the equity value process and a decoupling field corresponding to an adjoint process, respectively.
\item
We analyse by way of detailed numerical studies the structure of the central agent's optimal strategy in different market environments, and the minimal losses that are attained under optimal strategies with varying cost (see also Section \ref{sec:3a}).
%\item
%We examine a variant where the central agent  minimises a quadratic target with terminal costs depending on the final state of the equity process. In this case we are able to derive a semi-analytic solution which we analyse also numerically (see Section \ref{sec:4}).
\end{itemize}

%%%%%%%%%%%%%%%%%%%%%%%%%%%%%%%%%%%%%%%%%%%%%%%%%%%%%%%%%%%%%%%%%%%%%
%%%%%%%%%%%%%%%%%%%%%%%%%%%%%%%%%%%%%%%%%%%%%%%%%%%%%%%%%%%%%%%%%%%%%

\section{Convergence to a mean-field limit}
\label{sec:convergence}

In this section, we show the existence of a minimising strategy for the central agent's objective function in the mean-field limit,
as well as convergence of the $N$-agent control problem.

\subsection{The model setup}

We fix a measure $\nu \in \mathcal{P}([0,\infty))$ and define a reference probability space to be a tuple $\mathscr{S} = (\Omega, \mathcal{F}, (\mathcal{F}_t)_{t\geq 0},\mathbb{P})$ such that $\mathscr{S}$ supports a Brownian motion that is adapted to $(\mathcal{F}_t)_{t\geq0}$ and there is a $\mathcal{F}_0$-measurable random variable $X_{0-}$ with $\law(X_{0-}) = \nu$. Note that with this definition, $X_{0-}$ is independent of $B$ by construction. We endow the space

\begin{equation}
S_T:=\{f \in L^2([0,\infty))~|~ 0 \leq f \leq b_{\text{max}} ~\text{a.e.},~ f_{\vert_{(T,\infty)}} = 0\}
\end{equation}
with the topology of weak convergence %\footnote{in the sense of functional analysis} 
in $L^2([0,\infty))$. Since $S_T$ is bounded in the $L^2([0,\infty))$-norm and weakly closed, $S_T$ is a compact Polish space. We then define the space of admissible controls
\begin{equation}\label{eq:defadmissiblecontrols}
\mathcal{B}_T := \{\beta \text{ is } (\mathcal{F}_{t})_{t \geq 0}-\text{progressively measurable}~|~ \mathbb{P}(\beta \in S_T) = 1 \}.
\end{equation}
Note that the space of admissible controls $\mathcal{B}_T$ as well as the objective functional $J$ as defined in \eqref{eq::objective} depend implicitly on the choice of stochastic basis $\mathscr{S}$. We will sometimes write $\mathcal{B}_T{(\mathscr{S})}$ or $J^{(\mathscr{S})}$ when we wish to emphasize this dependence. To be able to guarantee existence of optimizers and to make the optimization problem independent of the choice of stochastic basis $\mathscr{S}$, we will consider the relaxed optimization problem 

\begin{equation}\label{eq:objectiverelaxed}
V_\infty := \inf_{\mathscr{S}} \inf_{\beta \in \mathcal{B}_T (\mathscr{S})} J^{(\mathscr{S})}(\beta),
\end{equation}
as is standard in the stochastic optimal control literature (see e.g.\ \cite{fleming2006controlled}). 
We say that $(X_{0-},B,\beta,\Lmin)$ solve problem \eqref{eq::controlled}-\eqref{eq::controlled2} on $\mathscr{S}$ if $\mathscr{S} = (\Omega, \mathcal{F}, (\mathcal{F}_t)_{t\geq 0},\mathbb{P})$ is a reference probability space  with Brownian motion $B$ and initial
condition $X_{0-}$ such that $\beta \in \mathcal{B}_{T}^{(\mathscr{S})}$ and \eqref{eq::controlled}-\eqref{eq::controlled2}
holds $\mathbb{P}$-almost surely.

Note that it is not clear a priori that the process $X$ given in \eqref{eq::controlled} is well-defined. Indeed, it is known that the McKean--Vlasov problem \eqref{eq::init1} and \eqref{eq::init2} may admit more than one solution, and it is not known that physical solutions
exist for general $\beta \in \mathcal{B}_T$ , although it is known for $\beta$ of the special form $b(t,X_t)$, where $b$ is Lipschitz (see e.g.\ \cite{DIRT,LS}). To  pin down a meaningful solution concept, we therefore rely on the notion of \emph{minimal solutions} as defined in \eqref{eq:defminimal}.  By the results of \cite{CRS}, we know that minimal solutions of the uncontrolled system are physical whenever the initial condition $X_{0-}$ is integrable.

%We proceed to show well-posedness of minimal solutions.
Throughout the following sections $\mathcal{P}(E)$ always denotes the set of probability measures on a Polish space $E$ which we endow with the Lévy-Prokhorov metric, i.e., convergence of probability measures is to be understood in the (probabilistic) weak sense. For function spaces etc.~we apply rather standard notation and refer to Section \ref{sec:app_notation} for more details.

\subsection{Well-posedness of minimal solutions for general drift}

We fix the reference probability space $\mathscr{S}$ and show that minimal solutions exist for any $\beta \in \mathcal{B}_T$. Define the operator $\Gamma$ for a \cadlag function $\ell$ and $\beta \in \mathcal{B}_T$ as
\begin{equation}\label{eqdef:Gammabeta}
\left\{
\begin{aligned}
X_t^{\ell}(\beta) &= X_{0-} + \int_{0}^{t}\beta_s \, ds + B_t - \alpha\ell_t, \\ 
\tau_{\beta}^{\ell} &= \inf\{t \geq 0: X_t^{\ell} (\beta)\leq 0 \}, \\
\Gamma[\ell,\beta]_t &= \P{\tau_{\beta}^{\ell} \leq t}.
\end{aligned}\right.
\end{equation}
Note here that $(X^\ell(\beta),\tau_{\beta}^\ell,\ell)$ solves \eqref{eq::controlled} if and only if $\ell$ is a fixed-point of $\Gamma[\cdot,\beta]$. We next introduce a function space that is mapped to itself by $\Gamma[\cdot,\beta]$: Set
\begin{equation}\label{eq:DistFunctiondef}
M := \{ \ell \colon \overline{\R} \rightarrow [0,1]~|~  \ell \text{ \cadlag and increasing, }~\ell_{0-} = 0,~\ell_{\infty} = 1\},
\end{equation}
where $\overline{\R}$ is the extended real line. Note that for $\ell \in M$, $\ell$  defines a cumulative distribution function of a probability measure on $[0,\infty]$.
%$\mathcal{P}([0,\infty])$. 
Therefore, equipping $M$ with the topology of weak convergence, i.e., we have that $\ell^n \to \ell$ in $M$ if and only if $\ell_{t}^n \to \ell_t$ for all $t \in [0,\infty]$ that are continuity points of $\ell$, we obtain that $M$ is a compact Polish space. As in the uncontrolled case, $\Gamma[\cdot,\beta]$ is continuous on $M$.

\begin{theorem}
For any $\beta \in \mathcal{B}_T$, the operator $\Gamma[\cdot,\beta]\colon M \rightarrow M$ is continuous. Furthermore, there is a (unique) minimal solution to \eqref{eq::controlled}, and it is given by
\begin{equation}
\underline{\Lambda}(\beta) = \alpha\lim_{k\to\infty} \Gamma^{(k)}[0,\beta].
\end{equation}
\begin{proof}
Using Lemma \ref{thm:crossingproperty}, this follows from Theorem 2.3 in \cite{CRR}.
\end{proof}
\end{theorem}

\subsection{Existence of an optimal control}
% We equip $\mathcal{B}_T$ with the topology induced by convergence in distribution on $S$, i.e., we say
% that $\beta^n \to \beta$ in $\mathcal{B}_T$ if $\operatorname{law}(\beta^n) \to \operatorname{law}(\beta)$ in $\mathcal{P}(S_T)$, where $\mathcal{P}(S_T)$ is furnished with the weak topology\footnote{in the sense of probability theory}. Phrased differently, we say that $\beta^n \to \beta$ in $\mathcal{B}_T$, if for any $F \in C_b(S_T;\mathbb{R})$ we have $\mathbb{E}(F(\beta^n)) \to \mathbb{E}(F(\beta))$.

A key step in proving existence of an optimizer is to show that sequences of
solutions to \eqref{eq::controlled}-\eqref{eq::controlled2} are compact in a certain sense and that their 
cluster points are solutions of \eqref{eq::controlled}-\eqref{eq::controlled2}.
This is the content of the next theorem. 

\begin{theorem}\label{thm:lambdaconvergence}
Let $(X_{0-}^n,B^n,\beta^n,\Lambda^n)$ solve \eqref{eq::controlled}-\eqref{eq::controlled2} on 
$\mathscr{S}^n$. Then, after passing to subsequences if
necessary, there is a reference probability space $\mathscr{S}$ such that $(X_{0-},B,\beta,\Lambda)$ solve \eqref{eq::controlled}-\eqref{eq::controlled2}
on $\mathscr{S}$ and it holds that $\law(\beta^n) \to \law(\beta)$ in $\mathcal{P}(S_{T})$ and $\frac{1}{\alpha} \Lambda^n \to \frac{1}{\alpha} \Lambda$ in $M$.
\end{theorem}
\begin{proof}
See Section \ref{app:existence} in the Appendix.
\end{proof}

\begin{remark}
Note that in Theorem \ref{thm:lambdaconvergence}, we do not assume that either the $\Lambda^n$ or $\Lambda$ are minimal solutions. At this point, we do not know how to prove that $\Lambda$ is minimal if all $\Lambda^n$ are minimal. Stability of the minimal solution is an open question (cf. \cite[Conjecture 6.10]{CRS}). For the proof of the subsequent theorem where we prove existence of an optimizer to \eqref{eq:objectiverelaxed}, formulated with the minimal solution, this however does not matter.
\end{remark}

Next, we prove that the infinite-dimensional problem \eqref{eq:objectiverelaxed} admits an optimizer.

\begin{theorem}\label{thm::optimizerexists}
There is an optimizer of \eqref{eq:objectiverelaxed}, i.e., there is a stochastic basis $\mathscr{S}^\star$ and $\beta^\star \in \mathcal{B}_T (\mathscr{S}^{\star})$ such that
\begin{align*}
V_{\infty} = \inf_{\mathscr{S}} \inf_{\beta \in \mathcal{B}_T(\mathscr{S})} J^{(\mathscr{S})}(\beta) = J^{(\mathscr{S}^\star)}(\beta^\star).
\end{align*} 
\end{theorem}
\begin{proof}
Let $(X_{0-}^n,B^n,\beta^n,\Lambda^n)$ be solutions 
to \eqref{eq::controlled}-\eqref{eq::controlled2}
on $\mathscr{S}^n$ such that  ${J^{(\mathscr{S}^n)}(\beta^n) \leq V_\infty + \frac{1}{n}}$. By Theorem \ref{thm:lambdaconvergence},
after passing to subsequences if necessary, there is 
a reference probability space $\mathscr{S}^\star$ such 
that $(X_{0-}^\star,B^\star,\beta^\star,\Lambda^\star)$
solves \eqref{eq::controlled}-\eqref{eq::controlled2} on
$\mathscr{S}^\star$ and $\law(\beta^n) \to \law(\beta^\star)$ holds in $\mathcal{P}(S_T)$ as well as
$\frac{1}{\alpha} \Lambda^n \to \frac{1}{\alpha} \Lambda^\star$ in $M$.

In the following, we simply write $J$ instead of $J^{(\mathscr{S}^\star)}$ etc. By construction, we have $J(\beta^n) \leq V_\infty + \frac{1}{n}$ and hence $\liminf_{n\to\infty} J(\beta^n) \leq V_\infty$.
It is clear that $V_\infty \le \gamma < \infty$ by \eqref{eq:objectiverelaxed} and \eqref{eq::objective},
as $J$ attains a value less than or equal to $\gamma$ for $\beta=0$. 
We proceed to show that $J(\beta^\star) \leq \liminf_{n\to\infty} J(\beta^n)$.\\

Since the functional $b \mapsto \int_{0}^{T} b_s \;ds$ is
continuous and bounded on $S_T$, it follows that
\begin{equation}
\lim_{n\to\infty} \mathbb{E}^n\left[\int_{0}^{T} \beta_s^n\, ds \right] = \mathbb{E}^\star\left[\int_{0}^{T}\beta_s^\star\,ds\right].
\end{equation}
The Portmanteau theorem implies that
\begin{equation}
\liminf_{n\to\infty}\Lambda_{T-}^{n} \geq \Lambda_{T-}^{\star},
\end{equation}
and since $\Lambda^\star$ solves \eqref{eq::controlled}-\eqref{eq::controlled2} with drift $\beta^\star$ and $\underline{\Lambda}(\beta^\star)$ is the minimal solution on $\mathscr{S}^\star$ with drift $\beta^\star$, it follows that $\underline{\Lambda}_{T-}(\beta^\star) \leq \Lambda_{T-}^\star$ which is equivalent to $\underline{L}_{T-}(\beta^\star) \leq L_{T-}$, which concludes the proof.
\end{proof}

\subsection{Properties of the \NplayerGame}

We describe the \NplayerGame\ mentioned in the introduction in more detail. We consider a stochastic basis $\mathscr{S}_N = (\Omega_N, \mathcal{F}^N, (\mathcal{F}_t^N)_{t \geq 0}, \mathbb{P}_N)$ supporting certain exchangeable random variables, defined as follows.

\begin{definition}
Set $X^N:=(X^{1,N},\dots,X^{N,N})$, where the $X^{i,N}$ are random variables taking values in some space $E$. We say that $X^N$ is $N$-exchangeable, if 
\begin{align*}
\law(X^N)=\law ((X^{\sigma(1),N},X^{\sigma(2),N},\dots,X^{\sigma(N),N})),
\end{align*}
for any permutation $\sigma$ of $\{1,\dots,N\}$. We say that $\beta^N$ is $\mathscr{S}_N$-exchangeable if the vector $(X_{0-}^{i,N},B^{i,N},\beta^{i,N})_{i=1}^{N}$ is $N$-exchangeable under $\mathbb{P}_N$.
\end{definition}

The stochastic basis $\mathscr{S}_N$ is supposed to support an $N$-dimensional Brownian motion $B^N$ and an $N$-exchangeable, $\mathcal{F}_0^N$-measurable random vector $X_{0-}^N$. The particles in the system then satisfy the dynamics
\begin{align}\label{eq::controlledNplayer}
X_{t}^{i,N} := X_{0-}^{i,N} + \int_{0}^{t} \beta_s^{i,N}\,ds +  B_t^{i,N} - \L_t^N,
\end{align}
where $\beta^N$ is $\mathscr{S}_N$-exchangeable, and $\L_t^N = \frac{\alpha}{N} \sum_{i=1}^{N} \mathds{1}_{\{\tau^{i,N} \leq t\}}$, where $\tau_{i,N} := \inf\{t\geq 0: X_{t}^{i,N} \leq 0\}$. In analogy to the infinite-dimensional case, we denote $L^N := \frac{1}{\alpha} \L^N$. We then consider the set of admissible controls 
\begin{equation}\label{eq:defadmissiblecontrolsN}
\mathcal{B}_T^N := \{\beta^N \text{ is $\mathscr{S}_N$-exchangeable, } (\mathcal{F}_{t}^N)_{t \geq 0}-\text{progressively measurable}~|~ \mathbb{P}(\beta^{1,N} \in S_T) = 1 \}.
\end{equation}
The same examples as in the uncontrolled case show that solutions to \eqref{eq::controlledNplayer} are not unique in general (cf. Section 3.1.1 in \cite{DIRT}). 
Therefore, in \cite{DIRT2},  physical solutions are introduced. Similarly to the infinite dimensional case we can also consider minimal solutions.  We call a solution $\underline{\L}^N$ to \eqref{eq::controlledNplayer} minimal, 
if for every solution $\L^N$ to \eqref{eq::controlledNplayer}
\begin{align*}
\underline{\L}_t^N \leq \L_t^N, \quad t \geq 0,
\end{align*}
holds almost surely.
The same argument as in \cite[Lemma 3.3]{CRS} shows that the notions of physical and minimal solution are equivalent in the \NplayerGame.
In analogy to the infinite-dimensional case, we introduce the operator 

\begin{equation}\label{eqdef:GammaNbeta}
\left\{\begin{aligned}
X_{t}^{i,N}[\mathsf{L},\beta^N] &= X_{0-}^{i,N} + \int_{0}^{t} \beta_s^{i,N}\,ds + B_{t}^{i,N} -\alpha\mathsf{L}_t \\
\tau_{i,N}[\mathsf{L},\beta^N] &= \inf\{t \geq 0: X_{t}^{i,N}[\mathsf{L},\beta^N] \leq 0\} \\
\Gamma_N[\mathsf{L},\beta^N]_t &= \frac{1}{N}\sum_{i=1}^{N} \ind{\left\{\tau_{i,N}[\mathsf{L},\beta^N]\leq t\right\}}, 
\end{aligned}\right.
\end{equation}
where $\mathsf{L}$ is some \cadlag process adapted to the filtration generated by $B^N$. 
We will often simply write $\Gamma_N[\mathsf{L}]$ instead of $\Gamma_N[\mathsf{L},\beta^N]$. The statements are then meant to hold for arbitrary, fixed $\beta^N$.
An important property is that $\Gamma_N[\cdot,\beta^N]$ is monotone in the sense that 
$$
\mathsf{L}^1_t \leq \mathsf{L}^2_t, \quad t \geq 0 \quad \implies \quad \Gamma_N[\mathsf{L}^1,\beta^N]_t \leq \Gamma_N[\mathsf{L}^2,\beta^N]_t, \quad t \geq 0.
$$
We then readily see by straightforward induction arguments that 
\begin{equation}\label{eq:GammaNiterationprop}
\alpha \Gamma_N^{(k)}[0] \leq \L^N, \quad \Gamma_N^{(k)}[0] \leq \Gamma_N^{(k+1)}[0], \quad k \in \N,
\end{equation}
holds almost surely, where $\L^N$ is any solution to the particle system
and $\Gamma_N^{(k)}$ denotes the $k$-th iterate of $\Gamma_N$. A similar argument as for the system without drift in \cite{CRS} shows that the iteration $(\Gamma_N^{(k)}[0])_{k\in\N}$ converges to the minimal solution after at most $N$ iterations.
\begin{lemma}\label{lem:miniteration}
For $N \in \N$, let $\Gamma_N$ be defined as in \eqref{eqdef:GammaNbeta}.
Then  $\underline{\L}^N:=\alpha\Gamma^{(N)}_N [0,\beta^N]$ is the minimal solution to the particle system with drift $\beta^N$ and the error bound
\begin{align}\label{eq:nerrorbound}
\|\alpha\Gamma^{(k)}_N[0,\beta^N] - \underline{\L}^N\|_{\infty} \leq\alpha \frac{(N-k)^+}{N}
\end{align}
holds almost surely. 
\end{lemma}
\begin{proof}
Analogous to the proof of Lemma 3.1 in \cite{CRS}.
\end{proof}

Roughly speaking, the next result says that limit points (in distribution) of solutions to the \NplayerGame\
converge (along subsequences) to solutions of the controlled McKean--Vlasov equation \eqref{eq::controlled}-\eqref{eq::controlled2}. By $D[-1, \infty)$ we denote here the space of càdlàg paths on $[-1, \infty)$ equipped with the $M_1$-topology.

\begin{theorem}\label{thm:finitedimconvergence} For $N \in \N$, let $(X^N,\beta^N,\L^N)$ be a solution to the particle system
\eqref{eq::controlledNplayer} on the stochastic basis $\mathscr{S}_N$ and define $\mu_N := \frac{1}{N} \sum_{i=1}^{N} \delta_{X^{i,N}}$.
Suppose that for some measure $\lawXnull \in \mathcal{P}(\R)$ we have
$$
\lim_{N\to\infty} \frac{1}{N}\sum_{i=1}^{N}\delta_{X_{0-}^{i,N}} = \lawXnull.
$$
Then there is a subsequence (again denoted by $N$) such that  $\law(\EmpM_N) \to \law(\EmpM)$ in $\mathcal{P}(\mathcal{P}(D([-1,\infty))))$, where 
%$\LawEmpM_{\infty}$-almost every $\mu \in \mathcal{P}(D([-1,\infty))$ 
$\mu$ coincides almost surely with the law of a solution process $X$ to the
McKean--Vlasov problem \eqref{eq::controlled}-\eqref{eq::controlled2} satisfying $\law(X_{0-})=\lawXnull$.
\end{theorem}

\begin{proof}
See Section \ref{app::finitedimconvergence}
\end{proof}

The next theorem shows that when we optimize the policy in the particle system and then take the limit of the resulting optimal values, we obtain the same value that we find by optimizing the infinite-dimensional version of the problem. 

\begin{theorem}\label{thm::optimizersconverge}
For $\kappa \in (0,1/2)$, define the value of a perturbed \NplayerGame\ as
\begin{align*}
V_N &:= \inf_{\mathscr{S}_N} \inf_{\beta^N \in \mathcal{B}(\mathscr{S}_N)} J_N(\beta^N), \quad J_N(\beta^N) := \mathbb{E}^{N}\left[\int_{0}^{T}\beta^{1,N}_s\,ds + \gamma \Lpertmin_{T-}^N(\beta^N)\right],
\end{align*}
where $\Lpertmin^N(\beta^N) := \frac{1}{\alpha} \LpertminLambda^N(\beta^N)$ and $\LpertminLambda^N(\beta^N)$ is the minimal solution of the \NplayerGame\ with drift $\beta^N$ as introduced in Lemma \ref{lem:miniteration} and perturbed initial condition
$\tilde{X}_{0-}^{i,N} = X_{0-}^{i}+N^{-\kappa}$ for $\kappa \in (0,1/2)$ and all $i=1,\ldots,N$. Then it holds that
\begin{equation}
\lim_{N\to\infty} V_N = V_{\infty}.
\end{equation}
\end{theorem}

\begin{proof}
\underline{Step 1}: We show the inequality $\liminf_{N\to\infty} V_N \geq V_{\infty}$. To that end, choose $\mathscr{S}_N$ and $\beta^N \in \mathcal{B}_T(\mathscr{S}_N)$ such that 
\begin{equation*} 
\mathbb{E}^{N}\left[\int_{0}^{T}\beta^{1,N}_s\,ds + \gamma\Lpertmin_{T-}^{N} (\beta^N)\right] \leq V_N + \frac{1}{N}.
\end{equation*}
Arguing as in the proof of Theorem \ref{thm:finitedimconvergence}, we see that $\EmpMext_N = \frac{1}{N} \sum_{i=1}^{N} \delta_{(X_{0-}^{i}+N^{-\kappa}+B,\beta^N,\Lpertmin^N (\beta^N))}$ is tight, 
and by Theorem \ref{thm:finitedimconvergence} converges to a limit $\EmpMext$ that is supported on the set of solutions to the McKean--Vlasov problem \eqref{eq::controlled}-\eqref{eq::controlled2}. By Skorokhod representation, we may assume that this happens almost surely on some stochastic basis $\mathscr{S}$. Since the map $(w,b,\ell) \mapsto \int_{0}^{T} b_s\,ds + \gamma\ell_{T-}$ is bounded and lower semicontinuous on $C([0,\infty)) \times S_T \times M$, Fatou's Lemma and the Portmanteau theorem imply
\begin{align*}
\liminf_{N\to\infty}\mathbb{E}^{N}\left[\int_{0}^{T}\beta^{1,N}_s\,ds + \gamma\Lpertmin_{T-}^N (\beta^N)\right] &= \liminf_{N\to\infty} \mathbb{E}\left[\int_{}\left(\int_{0}^{T} b_s\,ds + \gamma\ell_{T-} \right)~\mathrm{d}\EmpMext_N (w,b,\ell)\right] \\
& \geq \mathbb{E}\left[\liminf_{N\to\infty}\int_{}\left(\int_{0}^{T} b_s\,ds + \gamma\ell_{T-} \right)~\mathrm{d}\EmpMext_N (w,b,\ell)\right] \\
& \geq \mathbb{E}\left[\int_{}\left(\int_{0}^{T} b_s\,ds + \gamma\ell_{T-} \right)~\mathrm{d}\EmpMext (w,b,\ell)\right].
\end{align*}
Defining $\mathscr{S}(\omega) = (C([0,\infty))\times S_T \times M, \mathscr{B}(C([0,\infty))\times S_T \times M), \EmpMext(\omega))$, let $(w,b,\ell)$ denote the canoncial process on $\mathscr{S}$. By the arguments in the proof of Theorem \ref{thm:finitedimconvergence} we have that $w$ is Brownian motion under $\EmpMext(\omega)$ with respect to the filtration generated by $(w,b,\ell)$, and we see that $\mathscr{S}(\omega)$ is an admissible reference space for almost every $\omega$. Since $\EmpMext(\omega)$ corresponds to the law of a solution to the McKean--Vlasov problem \eqref{eq::controlled}-\eqref{eq::controlled2}, it follows that $\int_{}\left(\int_{0}^{T} b_s\,ds + \gamma\ell_{T-} \right)~\mathrm{d}\EmpMext (w,b,\ell) \geq V_{\infty}$ almost surely. We have therefore obtained 
\begin{equation}
\liminf_{N\to\infty} V_N \geq \liminf_{N\to\infty}\mathbb{E}^{N}\left[\int_{0}^{T}\beta^{1,N}_s\,ds +\gamma \Lpertmin_{T-}^N(\beta^N)\right] \geq V_{\infty}.
\end{equation}

\underline{Step 2}: We show that $\limsup_{N\to\infty} V_N \leq V_{\infty}$. Let $\mathscr{S}^\star$ be a probability space and $\beta^{\star} \in \mathcal{B}(\mathscr{S}^\star)$ be an optimizer attaining $V_{\infty}$, whose existence  was shown in Theorem \ref{thm::optimizerexists}. Let $\mathscr{S}_N^\star$ be the product space obtained by taking $N$ copies of $\mathscr{S}^\star$, and consider the (random) cost functional
\begin{equation}
c_N(b^N,\ell) = \int_{0}^{T}b^{1,N}_s\,ds + \gamma\ell_{T-} + \gamma N^{\kappa} \|\Gamma_N[\ell,b^N] - \ell\|_{\infty}, \quad b^N \in S_T^N, ~ \ell \in M,
\end{equation}
where $\|\cdot \|_{\infty}$ is the supremum norm on $[0,\infty)$. Let $\mathsf{M}(\mathscr{S})$ be the set of all $\mathcal{F}$-measurable random variables, defined on the stochastic basis $\mathscr{S}$, taking 
values in $M$ and consider the problem
\begin{equation}
\hat V_N := \inf_{\substack{\beta^N \in \mathcal{B}(\mathscr{S}_N^\star) \\ \mathsf{L} \in \mathsf{M}(\mathscr{S}_N^\star)}} \mathbb{E}_{\star}^{N}\left[c_N(\beta^N,\mathsf{L})\right]. 
\end{equation}
Letting $\betaNstar$ be the vector obtained by taking $N$ i.i.d.\ copies of $\beta^\star$, and choosing $\mathsf{L} \equiv \underline{L}(\beta^\star)$, which we abbreviate in the following with $\underline{L} := \underline{L}(\beta^\star)$, we obtain
\begin{align*}
\hat V_N \leq \mathbb{E}_{\star}^{N}\left[\int_{0}^{T}\beta^{\star}_s\,ds + \gamma \underline{L}_{T-} + \gamma N^{\kappa} \|\Gamma_N[\underline{L},\betaNstar] - \underline{L}\|_{\infty}\right].
\end{align*}
Noting that $\Gamma_N[\underline{L},\betaNstar]$ is the empirical cumulative distribution function of the i.i.d.\ random variables $\tau_{\beta^\star}^{i} := \inf\{t \geq 0: X_{0-}^{i} + \int_{0}^{t} \beta_s^{\star,i} \mathrm{d}s + B_{t}^i - \Lmin_t \leq 0 \}$, and that $\mathbb{P}(\tau_{\beta^\star}^i \leq t) = \underline{L}_t(\beta^\star)$, the same estimates as in Step 1 of the proof of Proposition 6.1 in \cite{CRS} show that 
\begin{align*} 
\lim_{N\to\infty} \mathbb{E}\left[N^{\kappa} \|\Gamma_N[\underline{L},\betaNstar] - \underline{L}\|_{\infty}\right] = 0.
\end{align*}
We have therefore shown that
\begin{align*} 
\limsup_{N\to\infty} \hat V_N \leq \mathbb{E}_{\star}^{N}\left[\int_{0}^{T}\beta^{\star}_s\,ds + \gamma\underline{L}_{T-}(\beta^\star)\right] = V_{\infty}.
\end{align*}
Now choose a sequence $\hat\beta^N \in \mathcal{B}_T(\mathscr{S}_N^\star), \mathsf{L}^N \in \mathsf{M}(\mathscr{S}_N^\star)$ 
such that $\mathbb{E}_{\star}^{N}[c(\hat\beta^N,\mathsf{L}^N)] \leq \hat V_N + \frac{1}{N}$.

Define the sequence of events $A^N = \{\omega \in \Omega^N \colon  \|\Gamma_N[\mathsf{L}^N,\hat\beta^N] - \mathsf{L}^N\|_{\infty} \leq N^{-\kappa}\}$ and set
$\hat{\mathsf{L}}^N = \mathsf{L}^N \ind{A^N} + \underline{L}^{N}(\hat \beta^N) \ind{\Omega^N \setminus A^N}$, where $\underline{L}^{N}(\hat \beta^N)$ is the
minimal solution on $\mathscr{S}^*_N$ with drift $\hat \beta^N$. With this choice, $\hat{\mathsf{L}}^N$ is in $ \mathsf{M}(\mathscr{S}^*_N)$ and satisfies  
\begin{align*}
    \|\Gamma_N[\hat{\mathsf{L}}^N,\hat\beta^N] - \hat{\mathsf{L}}^N\|_{\infty} \leq N^{-\kappa}, \quad
    \mathbb{E}_{\star}^{N}[c(\hat\beta^N,\hat{\mathsf{L}}^N)] \leq \mathbb{E}_{\star}^{N}[c(\hat\beta^N,\mathsf{L}^N)] \leq \hat V_N + \frac{1}{N}.
\end{align*}
Here we used that $\underline{L}^{N}(\hat \beta^N) \leq 1$ and $\Gamma_N[\underline{L}^{N}(\hat \beta^N),\hat{\beta}^N] = \underline{L}^{N}(\hat \beta^N)$.
This implies
\begin{align} \label{eq::Llowerbound}
\hat{\mathsf{L}}^N \geq \Gamma_N [\hat{\mathsf{L}}^N,\hat \beta^N] - N^{-\kappa}.
\end{align} 
Since $\hat{\mathsf{L}}^N \geq - N^{-\kappa}$, the monotonicity of $\Gamma_N$ implies that 
\begin{align}\label{eq::Llowerbound2}
\hat{\mathsf{L}}^N \geq \Gamma_N[-N^{-\kappa},\hat \beta^N] - N^{-\kappa} = \tilde{\Gamma}_N [0,\hat \beta^N] - N^{-\kappa},
\end{align}
where $\tilde{\Gamma}_N$ is defined as in \eqref{eqdef:GammaNbeta} with initial condition $\tilde{X}_{0-}^{i,N} := X_{0-}^{i} + \alpha N^{-\kappa}$. Combining \eqref{eq::Llowerbound2} with \eqref{eq::Llowerbound} and again using the monotonicity of $\Gamma_N$, we obtain
\begin{align*}
\hat{\mathsf{L}}^N \geq \Gamma_N[\tilde{\Gamma}_N [0,\hat \beta^N] - N^{-\kappa}, \hat \beta^N] - N^{-\kappa} = \tilde{\Gamma}_N^{(2)}[0,\hat \beta^N] - N^{-\kappa}.
\end{align*}
A straightforward induction then shows that $\hat{\mathsf{L}}^N \geq \tilde{\Gamma}_N^{(k)}[0,\hat \beta^N] - N^{-\kappa}$ for all $k \in \mathbb{N}$, and Lemma \ref{lem:miniteration} then yields that we have 
\begin{align*}
\hat{\mathsf{L}}^N \geq \Lpertmin^N(\hat \beta^N) - N^{-\kappa},
\end{align*}
where $\Lpertmin^N(\hat \beta^N)$ corresponds to the loss process associated to the particle system with initial condition $\tilde{X}_{0-}^{N}$.
This yields
\begin{align*}
V_N - \gamma N^{-\kappa} &\leq \mathbb{E}_{\star}^{N} \left[\int_{0}^{T} \hat\beta_s^{1,N}\,ds + \gamma\Lpertmin_{T-}^N(\hat \beta^N)  \right] - \gamma N^{-\kappa} \leq \mathbb{E}_{\star}^{N} \left[\int_{0}^{T} \hat\beta_s^{1,N}\,ds + \gamma\hat{\mathsf{L}}_{T-}^{N}\right] \\
&\leq \mathbb{E}_{\star}^{N} \left[c(\hat \beta^N,\hat{\mathsf{L}}^N)\right] \leq \hat V_N + \frac{1}{N}.
\end{align*}
Since we have already shown that ${\limsup_{N\to\infty} \hat V_N \leq V_{\infty}}$, this shows ${\limsup_{N\to\infty} V_N \leq V_{\infty}}$, which concludes the proof.
\end{proof}

\begin{remark}
We conjecture that the perturbation in the initial condition of the particle system in Theorem \ref{thm::optimizersconverge} is an artefact of our proof technique rather than a necessity. 
\end{remark}

\begin{theorem}\label{thm:epsilonoptimality}
Let $\mathscr{S}^\star$ be a probability space and $\beta^{\star} \in \mathcal{B}_T(\mathscr{S}^\star)$ be an optimizer attaining $V_{\infty}$.  Let $\mathscr{S}_N^\star$ be the product space obtained by taking $N$ copies of $\mathscr{S}^\star$ and let $\betaNstar$ be the vector obtained by taking $N$ i.i.d.\ copies of $\beta^\star$. Then, $\betaNstar$ is $\epsilon$-optimal for the particle system, i.e., for every $\epsilon > 0$, it holds that
\begin{align*}
J_N(\betaNstar) \leq V_N + \epsilon
\end{align*}
for $N$ sufficiently large.
\end{theorem}

\begin{proof}
Let $\epsilon > 0$ be given. Recall the notation introduced in Step 2 of the proof of Theorem \ref{thm::optimizersconverge}. Consider the problem
\begin{equation}
\bar V_N := \inf_{\mathsf{L} \in \mathsf{M}(\mathscr{S}_N^\star)} \mathbb{E}_{\star}^{N}\left[c_N(\betaNstar,\mathsf{L})\right]. 
\end{equation}
Proceeding as in Step 2 of the proof of Theorem \ref{thm::optimizersconverge}, it follows that $\limsup_{N\to\infty} \bar V_N \leq V_{\infty}$. By Theorem \ref{thm::optimizersconverge}, we have $\lim_{N \to \infty} V_N = V_{\infty}$, and therefore $\bar V_N \leq V_N + \epsilon / 3$ for $N$ large enough. Arguing as in Step 2 of the proof of Theorem $\ref{thm::optimizersconverge}$, we can find $\mathsf{L}^N \in \mathsf{M}(\mathscr{S}_N^\star)$
%(with $\mathsf{M}(\mathscr{S}_N^\star)$ denoting the set of all $(\mathcal{F}_t)$-adapted processes in $M$ on $\mathscr{S}_N^\star$)
such that $\mathbb{E}[c_N(\betaNstar, \mathsf{L}^N)] \leq \bar{V}_N + \epsilon/3$ and $ {\mathsf{L}}^N \geq \Lpertmin^N(\betaNstar) - N^{-\kappa}$ holds. Choosing $N$ large enough such that $\gamma N^{-\kappa} < \epsilon/3$, we obtain
\begin{align*}
J_N(\betaNstar)\leq \mathbb{E}_{\star}^{N} \left[\int_{0}^{T} \beta^\star_s\,ds + \gamma\hat{\mathsf{L}}_{T-}^{N}\right] + \gamma N^{-\kappa} \leq \bar{V}_N + 2\epsilon/3 \leq V_N +\epsilon.
\end{align*}
\end{proof}

%%%%%%%%%%%%%%%%%%%%%%%%%%%%%%%%%%%%%%%%%%%%%%%%%%%%%%%%%%%%%%%%%%%%%
%%%%%%%%%%%%%%%%%%%%%%%%%%%%%%%%%%%%%%%%%%%%%%%%%%%%%%%%%%%%%%%%%%%%%

\section{Numerical solution of the MFC problem} \label{sec:3a}

In this section, we present a numerical scheme for the central agent's mean-field control problem.
We directly compute the optimal feedback control by a policy gradient method 
(PGM; see Section~\ref{subsec:pol_grad} and  \ref{subsec:grad_desc}) applied to a
regularised version   of the dynamics and the objective function (see Section~\ref{sec:reg}).
The gradient is approximated by finite difference schemes for the density of the forward
process and a decoupling field for an adjoint process (see Section~\ref{subsec:num_imp}).
This will allow us to conduct parameter studies of the optimal strategies as well as the resulting losses and costs in Section~\ref{subsec:comp_anal}.

Recall from \eqref{eq::objective} and above the process $\underline{X}(\beta)$ corresponding to the minimal solution $\underline{\Lambda}(\beta)$, 
%analogously defined as in \eqref{eq:defminimal} but now for \eqref{eq::controlled},
and write the objective function as
\begin{eqnarray}
%\nonumber
	J(\beta) 
	%&=& 	\E \Big[ \int_0^T \beta_t \, dt \Big] + \gamma \, \pp\Big(\inf_{0\le s < T} \underline{X}_s(\beta) \le 0\Big) \\
%\nonumber	&=& \E \Big[ \int_0^T \beta_t \, dt + \gamma \,  \mathds{1}_{\{\underline{\tau}_{\beta} < T\}} \Big] \\
%	&=& %\E \Big[ \int_0^T \beta_t \, dt + \gamma \,  \mathds{1}_{\{\widehat{\underline{X}}_{T-} = 0\}} \Big], 
=	\E \Big[ \int_0^T (\beta_t + \gamma \dot{L}_t) \, dt \Big] ,
	\label{eq::objective2}
\end{eqnarray}
%\fbox{I changed $\ell$ to $L$ to be consistent with the rest of the paper}\\
where %$\widehat{\underline{X}}=\underline{X}_t \mathds{1}_{\{\tau> t\}}$ is the absorbed minimal solution process, $\tau$ the default time. 
$L_t = \pp\Big(\inf_{0\le s < T} \underline{X}_s(\beta) \le 0\Big)$, and derivatives of $L$ are defined in a distributional sense
if necessary.

In the case of regular solutions, the absorbed process  associated with \eqref{eq::controlled},
$\widehat{\underline{X}}=\underline{X}_t \mathds{1}_{\{\tau> t\}}$, for $\tau$ the hitting time of 0,
has a sub-probability density $p$ supported on $(0,\infty)$ and an atomic mass at $0$. Similarly as in  \cite[Theorem 1.1]{DNS},  it
satisfies the forward Kolmogorov
equation
\begin{equation}\label{p_PDE_beta}
	\begin{split}
		& \partial_t p + \partial_x (\beta  p)=\frac{1}{2}\partial_{xx}p+\dot{\Lambda}_t\partial_x p,\;\;x\ge0,\;\;t\in \T,  \\ 
	 	& p(0,x)=f(x),\;\; x\ge 0\quad\text{and}\quad p(t,0)=0,\;\;t\in\T, %, \\
%		& \dot{\Lambda}_t=\frac{\alpha}{2}\partial_x p(t,0),\;\;t\in[0,T]\quad\text{and}\quad\Lambda_0=0.
	\end{split}
\end{equation}
 where
\begin{equation}\label{p_bc}
	\begin{split}
		%\dot{\Lambda}_t=\frac{\alpha}{2}\partial_x p(t,0),\;\;t\in[0,T]\quad\text{and}\quad\Lambda_0=0.
		\Lambda_t = \alpha \Big(1 - \int_0^\infty p(t,x)\, dx \Big), \qquad t\in\T,
	\end{split}
\end{equation}
and where  $\T$ denotes the set of all $t\in[0,T]$ where $t\rightarrow L_t$ is differentiable.
If $ t\notin\T$, in particular in the event of a blow-up at $t$, we have the following jump condition for the solution of
\eqref{p_PDE_beta}, $p(t-,x) = p(t,x-\Lambda_t + \Lambda_{t-})$.

Assuming again regular enough solutions where we can take the derivative with respect to time of the equation ${L_t = 1- \int_{0}^{\infty} p(t,x)~dx}$, we find  that
\begin{align*}
\dot{L}_t = -\int_{0}^{\infty}\partial_{t}p~{d}x
= -\int_{0}^{\infty} \frac{1}{2} \partial_{xx}p~{d}x -\int_{0}^{\infty}
\alpha\dot{\L}\partial_{x}p~{d}x = \frac{1}{2} \partial_{x} p(t,0).
\end{align*}

%Using the Kolmogorov forward equation [cite Delarue at al]
%\begin{eqnarray*}
%bla,
%\end{eqnarray*}
%by integration by parts, we find for classical solutions that
%\[
%\dot{\ell}_t = \frac{1}{2} p_x(t,0).
%\]
%for the density $p$ of the is the absorbed minimal solution process $\widehat{\underline{X}}=\underline{X}_t \mathds{1}_{\{\tau> t\}}$, $\tau$ the default time.
Moreover, we can rewrite the controlled dynamics of $\widehat{\underline{X}}$ for $t<\tau$ as
\begin{align}\label{eq::see}
	dX_t &= (\beta_t - \alpha \dot{L}_t) \, dt + dB_t.
\end{align}

Note that the current control problem lies outside the standard MFC context due to the following three main aspects:
(i) the interaction through the boundary leads to a time derivative of the measure component, which makes the problem as written in \eqref{eq::see} (without replacing $\dot L_t$ by $\partial_x p(t,0)/2$) `non-Markovian';
(ii) the drift coefficient is non-Lipschitz in the measure component;
(iii) the dynamics are defined by an absorbed process, which moreover has an irregular drift coefficient (as $t\to L_t$ can be discontinuous in time).
We will address these points by a regularisation in the next section, which will subsequently allow us to apply a policy gradient method, which is inspired by \cite{RSZ22}.

\subsection{Regularisation}\label{sec:reg}

Denote by $\nu_t$ the law of $\widehat{\underline{X}}_t$ corresponding to $p(t,x) dx$ in the regular case where a density exists.
For (small) $h>0$,
we approximate $\shalf p_x(t,0)$ in terms of the measure $\nu_t$  by
\[
\dot{L}^h_t  =
\frac{1}{2} \int_{-\infty}^\infty (- \partial \phi^h)(x) \, \nu_t(dx)= \frac{1}{2}\langle - \partial \phi^h, \nu_t \rangle ,
\]
where $\phi^h(x)$ is a smooth approximation of the Dirac $\delta$ distribution with support in $[0,h]$ and where the bracket notation is used to denote the integral.  % should work similarly.
%{\color{blue}
%where $\phi^h(x)$ is the reflected heat kernel with mean $h$ and variance $h^2$,
%\[
%\phi^h(x) = \frac{c}{h} \left(\phi(x/h-1)-\phi(-x/h+1)\right),
%\]
%where $\phi$ is the standard normal density and $c = (\Phi(1)-\Phi(-1))^{-1}$, with the cumulative normal $\Phi$, a normalising constant.
%Any smooth approximation of $\delta$ with support in $[0,\infty)$ should work similarly.
%}

We then define a 
smooth function $\Phi^h: \mathbb{R} \rightarrow \mathbb{R}$ such that $\Phi^h(x)=0$ for
$x\le - \kappa h$ and $\Phi^h(x)=1$ for $x\ge 0$, for some $\kappa>0$,
%mollification of the Heaviside function $\mathcal{H}$ by
%\[
%\Phi^h(x) = \int_{-\infty}^\infty \mathcal{H}(x+y) \phi^h(y-h) \, dy,
%\]
and consider the dynamics
\begin{align}
\label{eqn:X_reg}
	dX_t &= a^h(X_t,\beta_t,\nu_t) \, dt + \sigma^h(X_t) \, dB_t,
\end{align}
where
\begin{eqnarray*}
a^h(x,b,\nu)   = %\Phi^h(x) (b - \alpha \dot{\ell}^h) = 
\Phi^h(x) \left(b +\text{\small $\frac{\alpha}{2}$} \langle \partial \phi^h, \nu \rangle  \right),
\qquad
\sigma^h(x) = \Phi^h(x).
\end{eqnarray*}
%{\color{red} \fbox{State phi and Phi in appendix for completeness.}}
For completeness, we give the specific $\phi^h$ and $\Phi^h$ used in our computations in
Appendix \ref{subsec:smooth}.
There, we also show the graphs of $\phi^h$ and its first two derivatives for the value 
$h=10^{-3}$, which is frequently used in our tests below.

Under the dynamics \eqref{eqn:X_reg},
the process does not get absorbed at 0, but once it crosses 0 from above its diffusion and drift coefficients decay rapidly so that with high probability it remains in the interval $[-\kappa h,0]$ (see Figure \ref{fig:control_bmax} for an illustration of the density of such a process). 

The reason why we  consider these modified dynamics is to cast the absorbed process  $\widehat{\underline{X}}_t$ into a standard McKean--Vlasov framework.
%where $\Phi^h$ is a mollification of the Heaviside function $\mathcal{H}$ and 
%\[
%\dot{\ell}^h = \frac{1}{2} \int_{-\infty}^\infty (- \partial \phi^h)(x) \, \nu(dx),
%\]
%where $\phi^h$ is a mollified $\delta$ distribution.
The objective function can also be rewritten and becomes
\begin{align}
%J(\beta) &= 
%\E \Big[ \int_0^T f^h(\beta_t,\nu_t)  \, dt \Big] =
%\E \Big[ \int_0^T \left(\beta_t + \text{\small %$\frac{\gamma}{2}$} \langle - \partial \phi^h, \nu_t %\rangle \right) \, dt \Big].J(\beta) &= 
J(\beta) &= 
\E \Big[ \int_0^T f^h(\beta_t,\nu_t)  \, dt \Big],
\qquad f^h(b,\nu)  =
b + \text{\small $\frac{\gamma}{2}$} \langle - \partial \phi^h, \nu \rangle.
\label{eqn:J_reg}
\end{align}

A crucial point here is that both the coefficients in the dynamics and the objective can be written in terms of $\beta$ and $\nu$ alone, i.e.\ without any time (or spatial) derivatives of the measure flow $(\nu_t)$.
Note also that  $(a^h, \sigma^h, f^h)$ satisfy the
differentiability assumptions made  \cite[Section 3]{acciaio2019extended}, which we shall need for the Fréchet differentiability of the function $F$ defined in \eqref{eq:FG} below.

Once we have the optimal control in feedback form $\beta^\star_t = \beta^\star(t,x)$ for $\beta^\star: [0,T] \times [0,\infty) \rightarrow [0,b_{\max}]$, and the associated density $p$ of $\widehat{\underline{X}}$, we can compute the optimal loss and cost pair as
\begin{eqnarray}
\label{opt_loss}
L^\star_T &=&  
\frac{1}{2} \int_0^T \int_{-\infty}^\infty (- \partial \phi^h)(x) \ p(t,x) \, dx dt, \\
\label{opt_cost}
C^\star_T &=& %\mathbb{E}\left[\int_0^T \beta_t \, dt \right] = 
\int_0^T \int_{-\infty}^\infty \beta^\star(t,x) \ p(t,x) \, dx dt.
\end{eqnarray}

%Due to the regularisation, therefore, the associated MFC problem falls in a standard setting. % \fbox{why only close?}

%%%%%%%%%%%%%%%%%%%%%%%%%%%%%%%%%%%%%%%%%%%%%%%%%%%%%%%%%%%%%%%%%%%%%
%%%%%%%%%%%%%%%%%%%%%%%%%%%%%%%%%%%%%%%%%%%%%%%%%%%%%%%%%%%%%%%%%%%%%

\subsection{Policy gradients}
\label{subsec:pol_grad}

We follow here in spirit the approach of \cite{RSZ22}.
We consider first a slightly more general form of the MFC  problem, written as a nonsmooth optimization problem
over the Hilbert space $\cH^2(\sR)$ of $\sR$-valued square integrable, progressively measurable processes,
\begin{equation}\label{eq:optimization}
\inf_{\beta \in \mathcal{H}^{2}(\R)} ( F(\beta)+G(\beta) ),
\end{equation}
with  the functionals  $F:\cH^2(\R)\to \sR$  and  
   $G:\cH^2(\R)\to \sR\cup\{\infty\}$  defined as follows:
for all $\beta\in \cH^2(\sR)$,
\begin{align}\label{eq:FG}
F(\beta)\coloneqq\E \bigg[
\int_0^T f^h(\beta_t,\nu_t) \, \mathrm{d}t %+g(X^{\alpha}_T,\mathcal{L}_{X^{\alpha}_T})
\bigg],
\quad
G(\beta)\coloneqq\E \bigg[
\int_0^T  g(\beta_t) \, \d t\bigg],
\end{align}
where $f^h$ is defined in \eqref{eqn:J_reg}.

%$X^\beta$ %\in \cS^2(\sR^d)$ 
%is the state process controlled by $\beta$, satisfying \eqref{eqn:X_reg}.

%\eqref{cforward}. 
%\begin{align}\label{eq::see2}
%	dX_t &= a(X_t,\beta_t,\nu_t) \, dt + \sigma^h(X_t) \, dB_t.
%\end{align}

%We consider here the absorbed process, but drop the hat for notational simplicity. Likewise, $\nu$ is the measure associated with the absorbed process,
%supported by a density on $(0,\infty)$ and with delta mass at 0.

The splitting of the objective function into $F$ and $G$ allows for a separate treatment of the smooth component $f^h$ and a non-smooth component $g$. 
We will use $g$ to incorporate the constraints on $\beta$, specifically,
$g(x) = 0$ for $x \in [0,b_{\max}]$ and $\infty$ outside.
It is clear that $G:\cH^2(\R)\to \sR\cup\{\infty\}$ is convex due to the convexity of $g$.

Assuming that $\nu_t$ lies in
the Wasserstein
space of probability measures on $\mathbb{R}$ with finite second moment, denoted by  $\mathcal{P}_2(\R)$, 
we introduce the Hamiltonian  $H: \R \times \R \times  \mathcal{P}_2(\R)\times \R \times \R \to \R$ by
%We now have the full Hamiltonian
\begin{equation}\label{eq:mfcE_full_hamiltonian} 
H(x,b,\nu,y,z)\coloneqq a^h(x,b,\nu) y + \sigma^h(x) z +  f^h(b,\nu),
\end{equation}
with 
\begin{equation}\label{eq:mfcE_full_hamiltonian_x} 
\partial_x H(x,b,\nu,y,z)=
\partial \Phi^h(x) \left(b + \text{\small $\frac{\alpha}{2}$} \langle \partial \phi^h, \nu \rangle  \right) y
+ \partial \Phi^h(x) z.
\end{equation}
Moreover, 
by 
\cite[Lemma 3.1]{acciaio2019extended}, 
 $F:\cH^2(\R)\to \sR$  is Fr\'{e}chet differentiable
and its derivative
 $\nabla F:\cH^2(\R)\to \cH^2(\R)$
 satisfies
 for all $\beta\in \cH^2(\sR)$,
 \bb%\label{eq:F_gradient}
 {(\nabla F)(\beta)_t}=
    (\partial_{b} H)(X^{\beta}_t,\beta_t, \nu_t, Y^{\beta}_t, Z^{\beta}_t) 
 + \tilde{\mathbb{E}}[ (\partial_{\nu} H)(\tilde{X}^{\beta}_t, \tilde{\beta}_t, \nu_t, \tilde{Y}^{\beta}_t, \tilde{Z}^{\beta}_t)(X^{\beta}_t)],
\ee
$\d t\otimes \d \PP$-a.e.
 Here, 
 $X^\beta$ %\in \cS^2(\sR^d)$ 
is the state process controlled by $\beta$, satisfying \eqref{eqn:X_reg}, and
 $(Y^{\beta}, Z^{\beta})$ are square integrable adapted adjoint processes  such that 
for all $t\in [0,T]$,  
\begin{align}\label{adjointa}
\begin{split}
\mathrm{d}Y^{\beta}_t&=
(-\partial_x H(X^{\beta}_t,\beta_t,\nu_t,Y_t^{\beta}, Z_t^{\beta})
-\tilde{\E}[(\partial_{\nu} H)(\tilde{X}^{\beta}_t, \tilde{\beta}_t,\nu_t,\tilde{Y}^{\beta}_t, \tilde{Z}^{\beta}_t)(X^{\beta}_t)] )\,\d t
%\\
%&\q 
+Z^{\beta}_t\, d W_t,
 \\
 Y^{\beta}_T& = 0.
 %(\partial_x g)(X^{\alpha}_T,\mathcal{L}_{X^{\alpha}_T})+\tilde{\E}[(\partial_{\mu} g)(\tilde{X}^{\alpha}_T,\mathcal{L}_{X^{\alpha}_T})(X^{\alpha}_T)].
 \end{split}
\end{align} 
Above and hereafter, we 
use  the tilde notation to denote an independent copy of a random variable as in \cite{acciaio2019extended}.

We now consider controls in feedback form, namely $\beta_t = \beta(t,X_t)$,
which determine $X^\beta_t$ as solution of
\begin{equation}\label{forward2}
\mathrm{d} X_t= a^h(X_t, \beta(t,X_t),\nu_t) %,\dot{\ell}^m_t)
\, dt +  \sigma^h(X_t) \, dW_t.
\end{equation}

Then a sufficiently smooth decoupling field $u$ such that $Y_t= u(t,X_t)$ and $Z_t=\sigma^h(X_t) \partial_x u(t,X_t)$ satisfies
%\begin{equation}\label{decoupling_pde}
%\partial_t u + \text{\small $\frac{1}{2}$}  \sigma^h(x)^2 \partial_x^2 u + a(x, \beta(t,x),\nu) \partial_x u = -
%\tilde{\E}[(\partial_{\nu} H)(\tilde{X}_t, \beta(t,\tilde{X}_t),\nu_t,u(t,\tilde{X}_t), \sigma^h(\tilde{X}_t) \partial_x u(t,\tilde{X}_t))(x)],
%\end{equation}
%where $u(T,\cdot) = 0$.

\begin{equation}
\begin{split}\label{decoupling_pde}
&\partial_t u + \text{\small $\frac{1}{2}$}  \sigma^h(x)^2 \partial_x^2 u + a^h(x, \beta(t,x),\nu) \partial_x u = 
-\partial_x H(x,\beta(t,x),\nu,u,\sigma^h(x) \partial_x u)\\
&\qquad \qquad \qquad \qquad \quad- \tilde{\E}[(\partial_{\nu} H)(\tilde{X}_t, \beta(t,\tilde{X}_t),\nu_t,u(t,\tilde{X}_t), \sigma^h(\tilde{X}_t) \partial_x u(t,\tilde{X}_t))(x)],
\end{split}
\end{equation}
with terminal condition $u(T,\cdot) = 0$.

\subsubsection*{Computation of gradient by decoupling fields}

In our application, we can express the right-hand side of \eqref{decoupling_pde} more explicitly. For
the Hamiltonian \eqref{eq:mfcE_full_hamiltonian} with $a^h$ and $f^h$ defined by
(\ref{eqn:X_reg}) and (\ref{eqn:J_reg}), respectively,
we have, by \cite[Section 5.2.2, Example 1]{CD:18}
\begin{eqnarray*}
(\partial_{\nu} H)(\tilde{X}_t, \tilde{\beta}_t,\nu_t,\tilde{Y}_t, \tilde{Z}_t)(X_t)
&=& (\partial_{\nu}a^h)(\tilde{X}_t, \tilde{\beta}_t,\nu_t)(X_t) \, \tilde{Y}_t
+ (\partial_{\nu} f^h)(\tilde{\beta}_t,\nu_t)(X_t) \\
&=& \text{\small $\frac{\alpha}{2}$} \Phi^h(\widetilde{X}_t) \partial^2 \phi^h(X_t) \tilde{Y}_t -
\text{\small $\frac{\gamma}{2}$} \partial^2 \phi^h(X_t), \\
% \end{eqnarray*}
%Hence,
%\begin{eqnarray*}
- \tilde{\E}[(\partial_{\nu} H)(\tilde{X}_t, \tilde{\beta}_t,\nu_t,\tilde{Y}_t, \tilde{Z}_t)(X_t)] &=& 
\text{\small $\frac{1}{2}$} 
(\gamma - \alpha \tilde{\E}[\Phi^h(\widetilde{X}_t) \tilde{Y}_t])
\, \partial^2 \phi^h(X_t) \\
 &=&
\text{\small $\frac{1}{2}$} 
\left(\gamma - \alpha \langle \Phi^h u(t, \cdot), \nu_t \rangle \right)
\partial^2 \phi^h(X_t).
\end{eqnarray*}
%{\color{red} \fbox{Does this look correct?}}

Consequently, for the decoupling field $u$, with $z=\sigma^h \partial_x u= \Phi^h\partial_x u$,
\begin{align*}
&\partial_t u + \text{\small $\frac{1}{2}$} \sigma^h(x)^2 \partial_x^2 u + a^h(x,\beta(t,x),\nu_t) \partial_x u
 = 
- \partial \Phi^h(x) (\beta(t,x) +\text{\small $\frac{\alpha}{2}$} \langle  \partial \phi^h, \nu_t \rangle) u \\
& \hspace{6.5 cm} -  \partial \Phi^h(x) \Phi^h(x)\partial_x u + \text{\small $\frac{1}{2}$} 
\left(\gamma - \alpha \langle \Phi^h u(t, \cdot), \nu_t  \rangle \right)
\partial^2 \phi^h(x),
\nonumber
\end{align*}
which can be re-written as
\begin{eqnarray}\nonumber
\partial_t u + \text{\small $\frac{1}{2}$} \, \partial_x\! \left(\Phi^h(x)^2 \partial_x u \right) + 
(\beta(t,x) +\text{\small $\frac{\alpha}{2}$} \langle  \partial \phi^h, \nu_t \rangle) \, \partial_x\!\left(\Phi^h(x)  u \right)
= \text{\small $\frac{1}{2}$} 
\left(\gamma \!-\! \alpha \langle \Phi^h u(t, \cdot), \nu_t  \rangle \right)
\partial^2 \phi^h(x).
\label{decoupling_pde_R}
\end{eqnarray}
%with the density $p$ of $X$.
As
$(\partial_{b} H)(x,b, \nu, y,z) = \Phi^h(x) y + 1$,
we obtain
\bb%\label{eq:F_gradient}
 (\nabla F)(\beta)(t,x)=
 \Phi^h(x) u(t,x)
 + 1 -
 \text{\small $\frac{1}{2}$} 
\left(\gamma - \alpha \langle \Phi^h u(t, \cdot), \nu_t  \rangle \right)
\partial^2 \phi^h(x),
\ee
where we will assume that $\nu_t$ has a density $p(t,\cdot)$ which satisfies
\begin{equation}%\label{density_pde}
\partial_t p + \partial_x \left(a^h(x, \beta(t,x), \nu_t) p \right) = 
 \frac{1}{2} \partial_x^2 (\sigma^h(x)^2 p).
\end{equation}

\subsection{A proximal policy gradient method (PGM)}
\label{subsec:grad_desc}

We now compute a sequence of approximations to the optimal control in feedback form, namely $\beta^m_t = \beta^m(t,X^m_t)$.
Following \cite{RSZ22}, we will carry out proximal gradient steps with $\beta^0$ given, e.g.\ zero, and thereafter, for step size $\tau>0$,
\begin{align}\label{eq:phi_psi}
\begin{split}
% \beta^{m+1}(t,x)&= \text{prox}_{\tau g} \left(\psi^m(t,x)-\tau 
% (\nabla F)(\psi^{m})(t,x) \right),  \\
%  \psi^{m+1}(t,x)&=\beta^{m+1}(t,x) + \frac{m}{m+3}(\beta^{m+1}(t,x)-\beta^m(t,x)),
 \beta^{m+1}(t,x)&= \text{prox}_{\tau g} \left(\beta^m(t,x)-\tau 
 (\nabla F)(\beta^{m})(t,x) \right),
 \end{split}
\end{align}
where
  $\prox_{\tau g}:\sR^k\to \sR^k$ is the proximal map of $\tau g:\sR\to \sR\cup\{\infty\}$ such that 
  $$\textrm{prox}_{\tau g}(b)=\arg\min_{z\in \sR}\left(\frac{1}{2}|z-b|^2+\tau g(z)\right),
  \quad a\in \sR,\tau>0.$$
  
  %In the present case, 
  For the considered $g$, an indicator function,
  $\prox$ is simply the projection onto $[0,b_{\max}]$,
  i.e.\ $\textrm{prox}_{\tau g}(b)= \min(\max(b,0),b_{\max})$.

Then a sufficiently smooth decoupling field $u^m$ such that $Y^m_t= u^m(t,X_t^m)$ satisfies
%\begin{equation}\label{decoupling_pde_1}
%\partial_t u^m + \text{\small $\frac{1}{2}$} \, \partial_x\! \left(\Phi^h(x)^2 \partial_x u^m \right) + a(x, \beta^{m}(t,x),\nu^m) \partial_x u^m = 
%\text{\small $\frac{1}{2}$} 
%\left(\gamma - \alpha \langle \Phi^h u^m(t, \cdot), \nu^m_t  \rangle \right)
%\partial^2 \phi^h(x),
%\end{equation}
\begin{equation}
\begin{split}\label{decoupling_pde_1}
&\partial_t u^m + \text{\small $\frac{1}{2}$} \, \partial_x\! \left(\Phi^h(x)^2 \partial_x u^m \right) + 
(\beta^m(t,x) +\text{\small $\frac{\alpha}{2}$} \langle  \partial \phi^h, \nu^m_t \rangle) \, \partial_x\!\left(\Phi^h(x)  u^m \right)
\\&= \text{\small $\frac{1}{2}$} 
\left(\gamma \!-\! \alpha \langle \Phi^h u^m(t, \cdot), \nu^m_t  \rangle \right)
\partial^2 \phi^h(x),
\end{split}
\end{equation}
where $\nu^m(dx) = p^m \, dx$ for the density $p^m$ %of a process $X^m$
that satisfies
\begin{equation}\label{density_pde}
\partial_t p^m + \partial_x \left(a^h(x,\beta^{m}(t,x), \nu^m) p^m \right) = 
\text{\small $\frac{1}{2}$}  \partial_x^2 \Phi^h(x)^2  p^m,
\end{equation}
and where
\[
a^h(x,\beta^{m}(t,x), \nu^m)  =\Phi^h(x)  \left(  \beta^{m}(t,x) + \text{\small $\frac{\alpha}{2}$} \int_{-\infty}^\infty \partial \phi^h(x) \, p^m(t,x) \, dx \right).
\]
Finally,
\bb\label{eq:F_gradient}
 {(\nabla F)(\beta^m)}(t,x)=
% \psi^m(t,x) 
\Phi^h(x) u^m(t,x)
 + 1 -
 \text{\small $\frac{1}{2}$} 
\left(\gamma - \alpha \langle \Phi^h u^m(t, \cdot), \nu^m_t \rangle \right)
\partial^2 \phi^h(x).
\ee

%%%%%%%%%%%%%%%%%%%%%%%%%%%%%%%%%%%%%%%%%%%%%%%%%%%%%%%%%%%%%%%%%%%%%
%%%%%%%%%%%%%%%%%%%%%%%%%%%%%%%%%%%%%%%%%%%%%%%%%%%%%%%%%%%%%%%%%%%%%

\subsection{Numerical implementation}
\label{subsec:num_imp}

We pick regularisation parameters $h,\kappa >0$ for $\phi^h$ and $\Phi^h$ defined as above. Then in the $m$-th iteration, we first solve numerically \eqref{density_pde} for $p^m$ and then \eqref{decoupling_pde} for $u^m$,
where $\nu^m$ is the measure with density $p^m$.
We use a semi-implicit finite difference scheme on a non-uniform mesh, as detailed below.

We define a numerical approximation on a time mesh $t_i = i \Delta t$, $i\in \mathbb{I}=
\{0,1,\ldots, N\}$, $\Delta t = T/N$
for a positive integer $N$.

We also define a non-uniform spatial mesh $(x_j)_{j\in \mathbb{J}}$ with 
$\mathbb{J}= \{0,1,\ldots,J\}$, for $x_0 = x_{\min} <0 $, $x_J = x_{\max}>0$.

In the following, we drop the iteration index $m$ and use instead superscript $i$ to denote the timestep of any function defined on the space-time mesh
and subscript $j$ its spatial index, in particular, for the numerical PDE solutions,
$p_j^i\approx p(t_i,x_j)$, $u_j^i\approx u(t_i,x_j)$.
We assume a feedback control $b_j^i = \beta(t_i,x_j)$ is defined on this mesh.

Starting with the forward equation \eqref{density_pde},
for each $x_j$ and $t_i$, we approximate the drift coefficient $a$ by
\begin{eqnarray}
%a_j^i = \Phi^h(x_j) \left(b_j^i + \text{\small $\frac{\alpha}{2}$} \sum_{k=0}^{J-1} \partial \phi^h(x_k) \, p_k^{i-1} \, (x_{k+1}-x_k) \right),
a_j^i = \Phi^h(x_j) \left(b_j^i - \alpha L^i \right)
\quad \text{ for } \quad
L^i = -
\text{\small $\frac{1}{2}$} \sum_{k=0}^{J-1} \partial \phi^h(x_k) \, p_k^{i-1} \, (x_{k+1}-x_k),
\label{num_loss}
\end{eqnarray}
and set $s_j^i = \Phi^h(x_j)^2$.
Then define a finite difference scheme by $p_j^0 = f(x_j)$, and for $i>0$,
\begin{eqnarray*}
\frac{p^i_j-p^{i-1}_j}{\Delta t} 
+ \frac{\max(a_{j}^i,0)  p_{j}^i - \max(a_{j-1}^i,0)  p_{j-1}^i}{x_{j}-x_{j-1}}
+ \frac{\min(a_{j+1}^i,0)  p_{j+1}^i - \min(a_{j}^i,0)   p_{j}^i}{x_{j+1}-x_{j}}
%\frac{\nu^i_j q^i_j - \nu^i_{j-1} q^i_{j-1}}{\Delta x}
= && \\ %\frac{1}{2} 
\frac{1}{x_{j+1}-x_{j-1}} 
\left(\frac{s^i_{j+1}  p^i_{j+1} - s^i_{j}  p^i_{j}}{x_{j+1}-x_{j}} -
\frac{s^i_{j}  p^i_{j} - s^i_{j-1}  p^i_{j-1}}{x_{j}-x_{j-1}}
\right),
\qquad 0 < j < J,&& \\
p^i_j = 0, \qquad \text{else}.&&
\end{eqnarray*}
This is an upwind scheme for the first order terms, taking the appearance of $p$ in $a$ explicit, but otherwise implicit.
The form of the scheme is chosen to be consistent with \eqref{density_pde} for non-uniform meshes, in particular where the
mesh size is piecewise constant.

For the adjoint equation \eqref{decoupling_pde_R}, with $p_j^i$ now given in addition to $b_j^i$, we first define the right-hand side,
\begin{eqnarray}
\label{eqn:rhs_disc}
r_j^i = \text{\small $\frac{1}{2}$}  \left(\gamma - \alpha   \sum_{k=0}^{J-1} \partial \Phi^h(x_k) \, u^{i+1}_k p_k^{i} \, (x_{k+1}-x_k) \right) \partial^2 \phi^h(x_j),
\end{eqnarray}
and then, with $u_j^N = 0$, we define for $i<N$
\begin{eqnarray*}
\frac{u^{i+1}_j-u^{i}_j}{\Delta t} 
+ \min(a_{j}^i,0) \frac{u_{j}^i - u_{j-1}^i}{x_{j}-x_{j-1}}
+ \max(a_{j}^i,0) \frac{u_{j+1}^i - u_{j}^i}{x_{j+1}-x_{j}} 
%\frac{\nu^i_j q^i_j - \nu^i_{j-1} q^i_{j-1}}{\Delta x}
- \qquad\qquad\qquad\qquad\qquad && \\ %\frac{1}{2} 
\frac{1}{x_{j+1}-x_{j-1}} 
\left(
s^i_{j+1/2} \frac{u^i_{j+1} -  u^i_{j}}{x_{j+1}-x_{j}}-
s^i_{j-1/2} \frac{u^i_{j} -  u^i_{j-1}}{x_{j}-x_{j-1}} 
\right) = r_j^i,
\qquad 0 < j < J,&& \\
u^i_j = 0, \qquad \text{else}.&&
\end{eqnarray*}

This allows us to compute the gradient on the same mesh, from \eqref{eq:F_gradient},
\begin{eqnarray*}
G_j^i = \Phi^h(x_j) u_j^i + 1 - r_j^i,
\end{eqnarray*}
and perform updates
%\begin{eqnarray*}
$b_j^i \leftarrow \min(\max(b_j^i - \tau G_j^i,0),b_{\max}).$

%We set $\Delta x = \kappa \Delta t$ for some fixed $\kappa$.

Finally, \eqref{opt_loss} is approximated by $L$ in \eqref{num_loss} and
\eqref{opt_cost} by
\begin{eqnarray}
\label{num_cost}
\frac{T}{N} \sum_{i=1}^N \sum_{j_0}^{J-1} b_j^i \ p_j^i \ (x_{j+1}-x_j), \qquad j_0 = \min\{j: x_j>0\}.
\end{eqnarray}

Let us remark that we do not have a convergence proof for this numerical scheme and it also seems out of reach due to the delicate interplay between the discretization and regularization parameters visible from Table \ref{table:mesh_convergence}. Nevertheless, for fixed $N$ and $h$ we can empirically show  convergence of the gradient iterations (see Figure \ref{fig:PGM_conv}), which then allows us to compute approximate optimal policies. In this sense our numerical
tests indicate at least qualitatively how the optimal policies look like. 
Note that a rigorous convergence proof of a similar policy gradient iteration method in the non-mean field regime has recently been provided in \cite{reisinger2022linear}.

\subsubsection*{Set-up and model parameters}

In the rest of the paper, we give illustrations of the model's suggested strategies and resulting loss behaviour in different market scenarios,
influenced by the interaction parameter $\alpha$, the risk aversion $\gamma$, the initial state $f$, and maximum cash injection rate $b_{\max}$.

In all examples, we choose a gamma initial density,
\begin{eqnarray}
f(x) = 1/\Gamma(k) \theta^{-k} x^{k-1} {\rm e}^{-x/\theta}, \qquad  x\ge 0.
\label{gamma_dens}
\end{eqnarray}
The parameters of the initial distribution could be calibrated to CDS spreads if they are traded (see \cite{bujok2012numerical}).
The default parameters we use are $k=2$, $\theta=1/3$, chosen to give a range of different behaviours by varying the other parameters.
In this case, $f$ is differentiable with $f(0+)=0$.
This choice implies that there are smooth solutions for a short enough time interval (see \cite{HLS, DNS}).
It also implies  (see \cite[Theorem 1.1]{HLS}) that a blow-up (of the unregularised system) is guaranteed to happen at some time for $\alpha > 2 \mathbb{E}[X_{0-}] = 2 k\theta = 4/3$.
%, this being the expectation of the initial value.
Conversely, it is known (see \cite[Theorem 2.2 and the comment below it]{LS:20})
that the condition
%\begin{align*}
$\alpha \| f\|_{\infty} < 1$
%\end{align*}
leads to the so-called weak feedback regime, where continuity of solutions always holds true. 

%We will consider various values of $\alpha$ around 1. %, where we will see that the uncontrolled 
A simple estimation of meaningful $\alpha$ from typical asset volatilities, recovery rates, and mutual lending as proportion of overall debt  is found in \cite{lipton2019semi},
suggesting possible values from 0.3 to possibly higher than 5. We shall conduct tests for $\alpha\in \{0.5, 1, 1.5\}$.
With $\| f\|_{\infty} \approx 1.1$, it is clear that a jump cannot occur for $\alpha=0.5$,
but is guaranteed for $\alpha = 1.5$ as then $2 \mathbb{E}[X_{0-}]<\alpha$.
The terminal time is chosen as $T=0.02$.
We find empirically that the uncontrolled system does not jump in this interval for $\alpha=1$ (although it may jump eventually), and does jump halfway through the interval for $\alpha=1.5$.
We have intentionally chosen an initial distribution where blow-ups can happen at such relatively short time scales to illustrate the different effects.
In our regularised version of the problem, this manifests in a smooth transition to high values of losses,
around 60\%, over a short period of time. 
We fix $b_{\max}=10$ at first, and investigate the effect of larger values later on. 

In the following, when not stated otherwise, we choose $\kappa = 1/10$ in the construction 
of $\Phi^h$ (see \ref{sec:reg} and Appendix \ref{subsec:smooth}), which was found a reasonable choice in our experiments.
As default, solutions are computed with $N=800$ timesteps and a non-uniform mesh on $[x_{\min},x_{\max}] = [-2,6]$
which is constructed as described below.
%and $p(0,\cdot)$ is the density of a Gamma distributions with parameters $k=2$, $\theta=1/3$.

\subsubsection*{Mesh convergence}

We first analyse the convergence of the finite difference approximations for fixed control.
In particular, we first choose $\beta = 0$. The interaction parameter is $\alpha=0.5$.

The mesh is chosen uniformly in the intervals $[x_{\min},-0.02]$, $[-0.02,0.05]$, $[0.05,x_{\max}]$, such that approximately 5\% of the points lie in the first interval, 45\% in the second, and 50\% in the third, and the total number $N_x$ of spatial mesh points is approximately $N \cdot (x_{\max}-x_{\min})/(8 T)$. This has the effect that the average mesh size is roughly eight times the time step size, which turns out a reasonable ratio in our numerical tests.

It is of crucial importance to have enough mesh points in the intervals $[-\kappa h,0]$ and $[0,h]$ to approximate the smoothed Heaviside function and the smoothed delta distribution with its first two derivatives. A strong local mesh refinement as above allows this while keeping the total computational complexity feasible. Notice for our choice above the local mesh size around zero is almost 100 times smaller than for larger $x$.

In Table \ref{table:mesh_convergence}, we report for a varying number of time-steps $N$ (and proportionally chosen $N_x$) and smoothing parameter $h$ %(with $c=1/10$ in $\Phi^h$)
the computed loss (columns 5--10, rows 3--8).
Let $L_N^h$ be the loss computed with $N$ time steps and parameter $h$.
Then from the table we conjecture convergence of  $L_N^h$ as $N\rightarrow \infty$ for fixed $h$, but divergence as $h\rightarrow 0$ for fixed $N$.

%\begin{table}[ht!]
%\centering
%\begin{tabular}{|c|c|c|c|c|c|c|c|c|}
%\hline
%&& $N$ & 100 & 200 & 400 & 800 & 1600 & 3200 \\  \hline
%&& $N_x$ & 3750 & 7500 & 15000 & 30000 & 60000 & 120000 \\ \hline \hline
%$h$ &$10^3 \cdot \theta_n$ & $\rho_n$ &&&&&& \\
%&1.9565 & -2.4255 &    0.5643  &  0.6430  &  0.7137  &  0.8324  &  4.7756  &       0 \\
%&-0.8066 & 1.3123 &    0.5663  &  0.6260  &  0.6693  &  0.7268  &  0.8384  &  5.1210 \\
%&-0.6147 & 1.8232 &    0.5655  &  0.6164  &  0.6481  &  0.6778  &  0.8223  &  -0.0336 \\
%&-0.3371 & 1.9438 &    0.5649  &  0.6118  &  0.6376  &  0.6589  &  0.6884  &  0.6096 \\
%&-0.1734 & --- &    0.5645  &  0.6096  &  0.6327  &  0.6486  &  0.6647  &  0.6893 \\
%&--- & --- &    0.5643  &  0.6085  &  0.6304  &  0.6440  &  0.6548  &  0.6680 \\  \hline
%&$10^2 \cdot \vartheta_h$ && 4.4143  &  2.1927  &  1.3541  &  1.0842  &   1.3224 & --- \\
%&$\varrho_h$ && 2.0132   & 1.6193  &  1.2489  &  0.8199 & --- & --- \\ \hline
%&CPU (sec) && 0.44 & 1.2 & 4.2 & 17 & 83 & 427 \\ \hline
%\end{tabular}
%\caption{Mesh convergence, $\alpha=1.5$ and $\gamma=0.1$.
%}
%\label{table:mesh_convergence}
%\end{table}

\begin{table}[ht!]
%\centering
\hspace{-0.6 cm}
\begin{tabular}{|r|r|c|c|c|c|c|c|c|c|c|}
\hline
&& & $10^3 \cdot h$ & $1$ & $2^{-1}$ & $2^{-2} $ & $2^{-3}$ & $2^{-4} $ & $2^{-5} $ & CPU \\ \hline \hline
$\frac{N}{10^2}$ & $N_x/10^3$ & $10^3 \cdot \theta_N$ & $\rho_N$ &&&&&&&   (s) \\  \hline  
$1$ & $3.75$ &1.956 & -2.42 &    0.5643  &  0.6430  &  0.7137  &  0.8324  &  4.7756  &       0 & 0.44 \\
$2$ & $7.5$ &-0.806 & 1.31 &    0.5663  &  0.6260  &  0.6693  &  0.7268  &  0.8384  &  5.1210 & 1.2\\
$4$ & $15 $ &-0.614 & 1.82 &    0.5655  &  0.6164  &  0.6481  &  0.6778  &  0.8223  &  %-
0.0336 & 4.2 \\
$8$ & $30$  &-0.337 & 1.94 &    0.5649  &  0.6118  &  0.6376  &  0.6589  &  0.6884  &  0.6096 & 17 \\
$16 $ & $60$ &-0.173 & --- &    0.5645  &  0.6096  &  0.6327  &  0.6486  &  0.6647  &  0.6893 & 83 \\
$32 $ & $120$ &--- & --- &    0.5643  &  0.6085  &  0.6304  &  0.6440  &  0.6548  &  0.6680 & 427 \\  \hline
&& &$10^2 \cdot \vartheta_h$ & 4.414  &  2.192  &  1.354  &  1.084  &   1.322 & --- & \\
\hline
&& &$\varrho_h$ & 2.01   & 1.61  &  1.24  &  0.81 & --- & --- & \\ \hline
%&CPU (sec) && 0.44 & 1.2 & 4.2 & 17 & 83 & 427 \\ \hline
\end{tabular}
\caption{Mesh convergence, losses, $\alpha=1.5$ and $\gamma=0.1$.
}
\label{table:mesh_convergence}
\end{table}

To investigate this more quantitatively, we report in the third and fourth columns $\theta_N = L_{2 N}^{h}-L_{N}^{h}$ and $\rho_N = \theta_N/\theta_{2 N}$, where $h = 10^{-3}$.
The fact that, for fixed $h$, the increments $\theta_N$ for successive mesh refinements decrease inversely proportional to $N$ is consistent with first order convergence in $1/N$ and $1/N_x$.
Conversely, we fix $N=3200$ and examine $\vartheta_h = L_{N}^{h/2}-L_{N}^{h}$ and $\varrho_h = \vartheta_{h}/\vartheta_{h/2}$ in the last two rows.
The behaviour indicates a decrease of first order  in $h$ as long as $1/N_x$ is small compared to $h$, but divergence thereafter.
Finally, the approximate computational times, reported in the last column, are approximately linear in $N N_x$ and independent of $h$.\footnote{Computations performed
using Matlab on a 2.8 GHz Intel Core i7 with 16 GB 1600 MHz DDR3.}

A similar behaviour is observed for the approximation of the cost and for different parameters, as shown in Appendix \ref{subsec:meshconv}.

\subsubsection*{Convergence of policy gradient iteration (PGM)}

Next, we analyse the convergence of the policy gradient iteration.
Here and thereafter, we will use a modification
whereby \eqref{eq:phi_psi} is evaluated for $x > h$, while $\beta^{m+1} = b_{\max}$ for $x\le h$.
As the occupation time of $[0,h]$ is small, the effect of this choice has a negligible effect on the expected cost in all cases.
We found that this modified iteration converged faster and more reliably in our numerical tests.

We monitor in each iteration the loss at time $T$ computed as in \eqref{num_loss}, and the expected cost, computed as in \eqref{num_cost}.
For $L^{(m)}$ and $C^{(m)}$ the terminal loss and total expected cost at the $m$-th iteration, respectively,
we plot in Figure \ref{fig:PGM_conv} the steps
$|L^{(m+1)} - L^{(m)}|$ and $|C^{(m+1)} - C^{(m)}|$. In these tests, the iteration terminates if either both of these quantities are smaller than $10^{-5}$ or 50 iterations are reached. %whatever comes earlier.

The left-hand plot in Figure \ref{fig:PGM_diff_gamma} shows the convergence for different values of $\gamma \in \{10^{-3}, 10^{-2}, 10^{-1}\}$.
The intermediate value of $\gamma$ has the largest absolute error, while the smallest $\gamma$ leads to the smallest one. In the latter case, the cost is very small due to
the very small penalty of losses. The asymptotic rate of convergence appears similar for all parameters considered.

In Figure \ref{fig:PGM_diff_alpha}, we analyse the effect of $\alpha$ on the convergence. The error is largest for the smallest of
$\alpha \in \{0.5, 1, 1.5\}$, while the error is smallest for $\alpha=1.5$, which is the case where a jump occurs in an uncontrolled setting and losses are the largest.

\begin{figure}[t!]
    \centering
    \begin{subfigure}[t]{0.49\textwidth}
        %\centering
            \hspace{-1cm}
        \includegraphics[height=0.75\textwidth, width=1.15\textwidth]{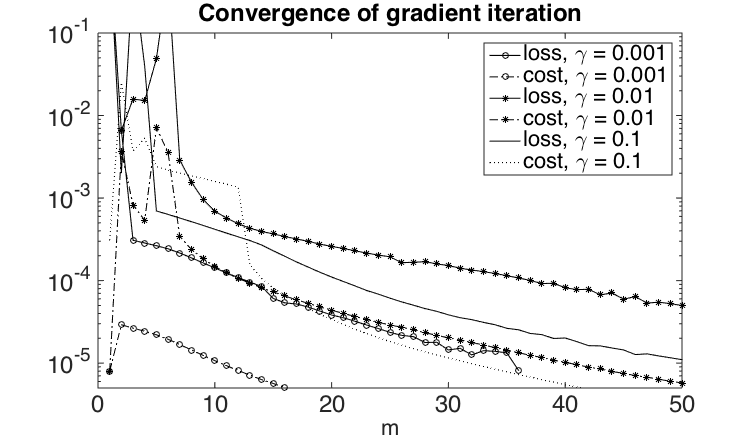}
        \caption{$\alpha=0.5$, varying $\gamma$, $N=800$}
        \label{fig:PGM_diff_gamma}
    \end{subfigure}%
    \hfill
    \begin{subfigure}[t]{0.49\textwidth}
        %\centering
        \hspace{-0.5cm}
        \includegraphics[height=0.75\textwidth, width=1.15\textwidth]{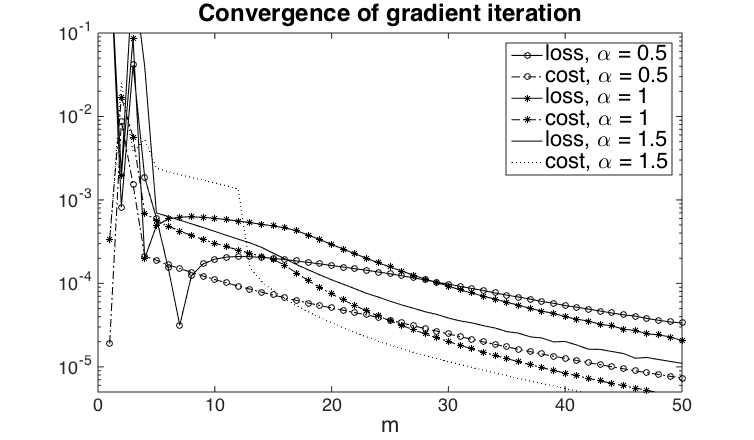}
        \caption{varying $\alpha$, $\gamma=1$, $N=800$}
       \label{fig:PGM_diff_alpha}
    \end{subfigure}
    \caption{Convergence of $C$ and $L$ in the PGM for varying $\gamma$ and $\alpha$. Shown are
    $|L^{(m+1)} - L^{(m)}|$ and $|C^{(m+1)} - C^{(m)}|$.}
    \label{fig:PGM_conv}
\end{figure}

Further parameter studies are given in the Appendix, where 
Figure \ref{subfig:PGM_N} establishes robustness of the convergence under mesh refinement and Figure \ref{subfig:PGM_tau} illustrates the effect of the step size.

In most situations, the number of iterations required for reasonable accuracy, i.e.\ a relative error below around $10^{-3}$ was between 10 and 30, so that for the chosen discretisation (with $N=800$ timesteps and mesh as chosen above) the computing time to solve the MFC problem was between 3 and 10 minutes on
the laptop as specified earlier.

%%%%%%%%%%%%%%%%%%%%%%%%%%%%%%%%%%%%%%%%%%%%%%%%%%%%%%%%%%%%%%%%%%%%%
%%%%%%%%%%%%%%%%%%%%%%%%%%%%%%%%%%%%%%%%%%%%%%%%%%%%%%%%%%%%%%%%%%%%%

\subsection{Computational analysis of central agent's strategy} \label{subsec:comp_anal}

We now move to an analysis of the optimal strategies, and the achievable pairs of costs and losses under the optimal and other strategies.

\subsubsection*{Analysis of the optimal strategy}

The policy gradient method produces directly an approximation to the optimal feedback control
$\beta^\star$. We found that an initialisation of the iteration with a function of the form
$\beta^0(t,x)=b_{\max}$ for $0<x<c$ and 0 elsewhere, for some $c>0$ large enough 
%{\color{red}\fbox{this is now a different $c$?}} 
so that the support
of $\beta^0$ covers the support of $\beta^\star$, produces more regular controls for small iteration numbers than a zero initialisation. The following plots were produced with $c=0.2$
and a tolerance $10^{-5}$ in the loss and cost (compare Figure \ref{fig:PGM_diff_gamma}).

We depict in Figure \ref{fig:cont_gamma} contours of the optimal feedback control $\beta^\star(t,x)$ for different $\gamma$.
As expected from the form of the Hamilton-Jacobi-Bellmann equation, the control is close to a `bang-bang' structure, i.e.\ a piecewise constant function where the control always takes one of the two extreme values,i.e. either $0$ or $b_{\text{max}}$. The two regions are  separated by a narrow strip where the
control transitions continuously. We conjecture this to be an effect of the numerical procedure,
which is designed for Lipschitz continuous feedback controls.

The (yellow) shaded region closest to $x=0$ is where $\beta^\star(t,x) \ge 0.95 \ b_{\max}$,
i.e.\ the central agent subsidises firms closest to default at or close to the maximum rate.
The white region furthest from $x=0$ is where $\beta^\star(t,x) \le 0.05 \ b_{\max}$,
i.e.\ the central agent does not subsidise firms with high reserves.
%We conjecture that the seemingly continuous transition of the optimal feedback control from one extreme of the control interval to the other 
%over a small strip is a result of the regularisation of the problem, while in the limit the control would exhibit a bang-bang structure.

%levels $\{0.05 \ b_{\max}, 0.95 \ b_{\max} \}$.
% \in \{0.1, 0.005, 0.001\}$.
%Shown are the contours of $\beta^\star(t,x)$ at levels $\{0.05 b_{\max}, 0.95 b_{\max} \}$.
For larger $\gamma$, here exemplified by $\gamma = 0.1$ in Figure \ref{fig:cont_gamma01}, the contribution of the loss to the objective is large enough for the central agent
to act for all $t$, for values in $x$ up to a decreasing curve in $t$. Close to the chosen end point, 
the effect of the control on the overall losses becomes negligible and does not justify the associated cost. In a sense this behaviour is an artifact of the finite observation interval.
\begin{figure}[t!]
%    \centering
    \begin{subfigure}[t]{0.32\textwidth}
        %\centering
            \hspace{-0.9cm}
        \includegraphics[height=0.75\textwidth, width=1.15\textwidth]{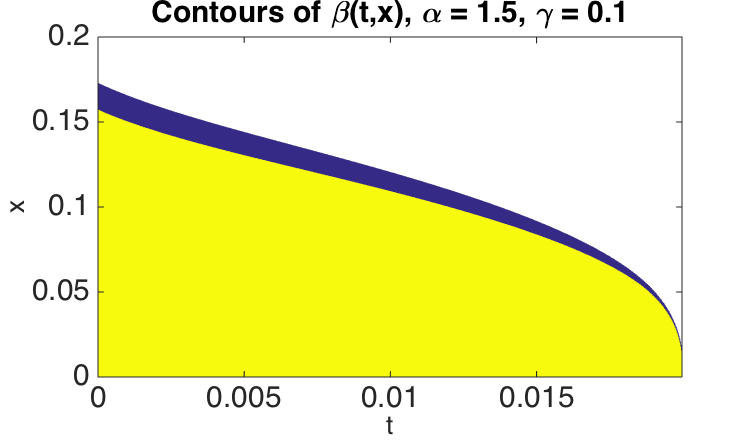}
        \caption{$\gamma=0.1$}
        \label{fig:cont_gamma01}
    \end{subfigure}
    \hfill
    \begin{subfigure}[t]{0.32\textwidth}
        %\entering
        \hspace{-0.9cm}
        \includegraphics[height=0.75\textwidth, width=1.15\textwidth]{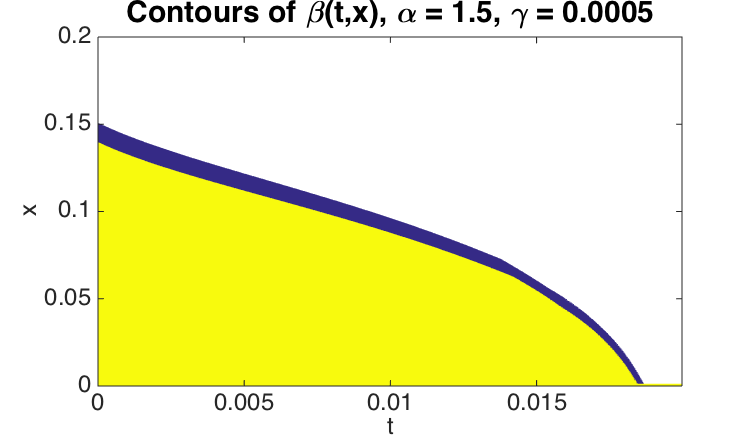}
        \caption{$\gamma=0.0005$}
       \label{fig:cont_gamma00005}
    \end{subfigure}
\hfill
    \begin{subfigure}[t]{0.32\textwidth}
        %\centering
        \hspace{-0.9cm}
        \includegraphics[height=0.75\textwidth, width=1.15\textwidth]{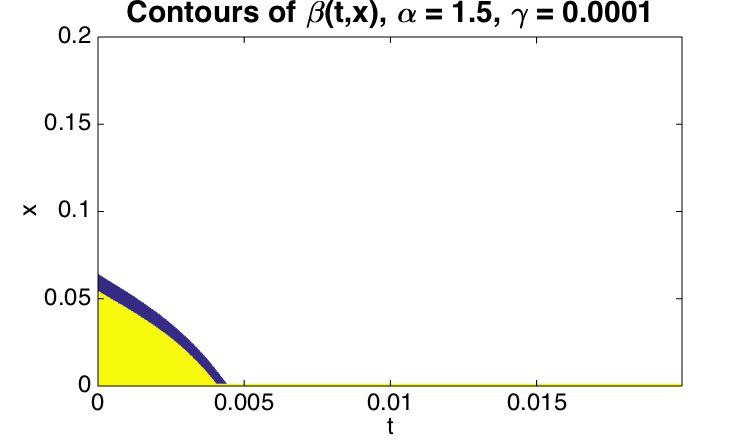}
        \caption{$\gamma=0.0001$}
       \label{fig:cont_gamma00001}
    \end{subfigure}     
    \caption{Contour plots of $(t,x) \rightarrow \beta^\star(t,x)$ for $\alpha=1.5$ and different $\gamma$.
    The white region is $\{\beta^\star\le 0.05 \ b_{\max}\}$, the (yellow) shaded region $\{\beta^\star \ge 0.95 \ b_{\max}\}$, the dark (blue) zone the transition.}
    \label{fig:cont_gamma}
\end{figure}

For smaller $\gamma$, in Figures \ref{fig:cont_gamma00005} and \ref{fig:cont_gamma00001}, the agent only acts (for sufficiently small states) up to a certain point in time 
and then does nothing.
Combining this with the plots of the resulting loss curves in Figure \ref{fig:loss_gamma}, a possible interpretation is that the agent seeks to delay the onset of the strongly contagious phase until it is no longer viable to do with a certain cost budget, depending on $\gamma$.
In particular, as visible from Figure \ref{fig:loss_gamma},
under the current optimization criterion
the jump is not avoided for $\gamma =0.001$.
We discuss at the end of the next subsection other strategies for avoiding jumps -- for the current optimisation criterion such a strategy is however not necessarily optimal.

%\begin{figure}[t!]
%    \centering
%        \includegraphics[height=0.38\textwidth, width=0.51\textwidth]{Plots/loss_combined_alpha15.png}    \hspace{-1 cm}     \hfill 
%        \includegraphics[height=0.38\textwidth, width=0.51\textwidth]{Plots/control_h0001_alpha15_gam01_b10_80_n50.png}  
%        \caption{Left: Loss function for $\gamma \in \{0.1, 0.005, 0.001\}$ (see also Figure \ref{fig:control_gamma}). Right: Control regions for different $b_{\max}$.}
%        \label{fig:control_bmax}
%\end{figure}

%\begin{minipage}{\columnwidth}
\begin{figure}[t!]
\centering
\begin{minipage}{.49\textwidth}
  %\centering
  \hspace{-0.5 cm}
  \includegraphics[height= 0.8\linewidth]{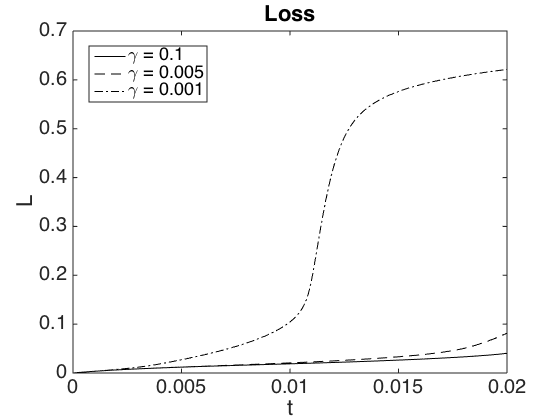}
  \captionof{figure}{Loss %function 
  for $\gamma \in \{0.1, 0.005, 0.001\}$.} % (see also Figure \ref{fig:cont_gamma}).}
  \label{fig:loss_gamma}
\end{minipage}%
\hfill
\begin{minipage}{.49\textwidth}
  %\centering
  \hspace{-0.2 cm}
  \includegraphics[height= 0.8\linewidth]{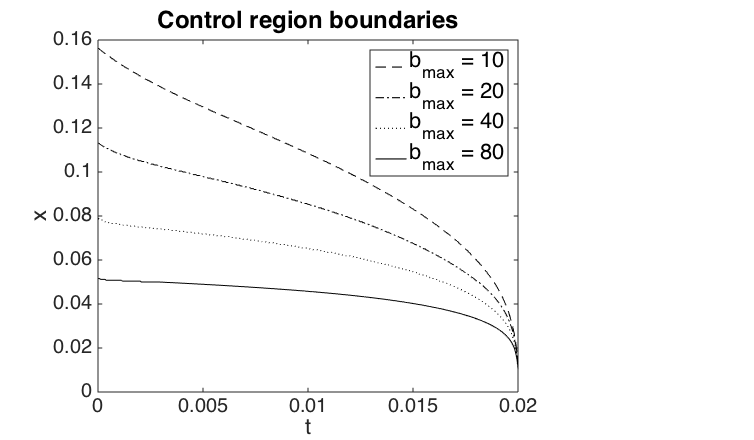}
  \captionof{figure}{Control regions for different $b_{\max}$.}
  \label{fig:control_bmax}
\end{minipage}
\end{figure}
%\end{minipage}

Next, we analyse the impact the upper bound of $b_{\max}$ on the control strategy.
We show only the $0.99\ b_{\max}$ level set for clarity in
Figure \ref{fig:control_bmax}, for different $b_{\max}$.
The region under this curve indicates where the agent controls at (or close to) the maximum rate.
The region shrinks as $b_{\max}$ increases, meaning that the agent is able to control the banks' equity process more effectively whenever it gets close to zero.

Lastly in this section, we analyse the behaviour of the PDE solutions $p$ and $u$ at different times and on different scales in $x$ around 0.
Figure \ref{fig:density} shows in the left panel the accumulation of probability mass 
in the interval $[-c h,0]$ due to the smooth truncation of the SDE coefficients, approximating
the absorption at $x=0$. As can be deduced from the right plot in Figure \ref{fig:density}, the area under the density for positive $x$ is thus reduced, but only
by a small amount in the current parameter setting (see Figure \ref{fig:loss_gamma} for the corresponding loss function with $\gamma=0.1$).
\begin{figure}[t!]
    \centering
     \includegraphics[height=0.38\textwidth, width=0.51\textwidth]{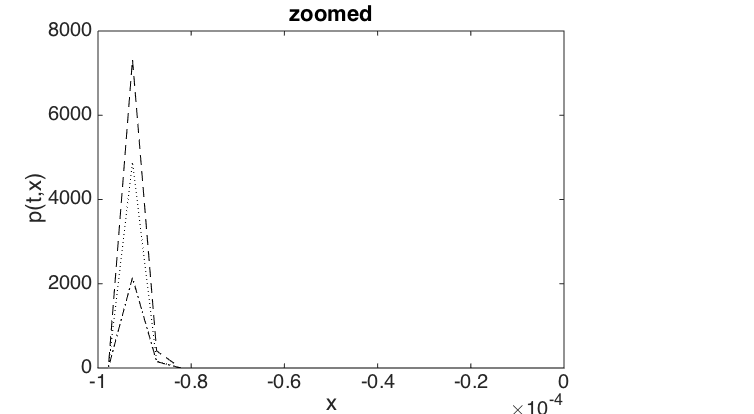}   \hspace{-1. cm}
        \includegraphics[height=0.38\textwidth, width=0.51\textwidth]{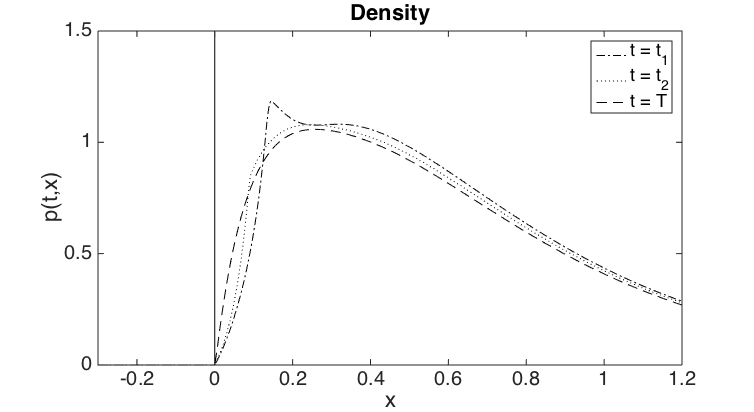}          
        \caption{Parameters $\alpha=1.5$, $\gamma = 0.1$. Left: Density $p(t,\cdot)$ for small negative $x$. Right: Density $p(t,\cdot)$ in macroscopic range.}
        \label{fig:density}
\end{figure}

In Figure \ref{fig:decoupling} we illustrate the behaviour of $u$ on different scales in $x$.
The left-most plot shows the range $[0,2 h]$, where $u$ attains large positive values;
in the middle plot, over $[0, 0.1]$, $u$ has moderate negative values;
the right-most plot is truncated below by $-1$, this being the threshold
which determines where the control is active.
This can be seen from \eqref{eq:F_gradient} in conjunction with \eqref{eq:phi_psi}:
for $x>h$, the gradient is $u+1$, so for a converged control we have
$\beta = b_{\max}$ where $u+1<0$ and $\beta = 0$ where $u+1>0$. From this the bang-bang structure of the control becomes also clear.

\begin{figure}[t!]
    \centering
    \includegraphics[height=0.38\textwidth, width=0.51\textwidth]{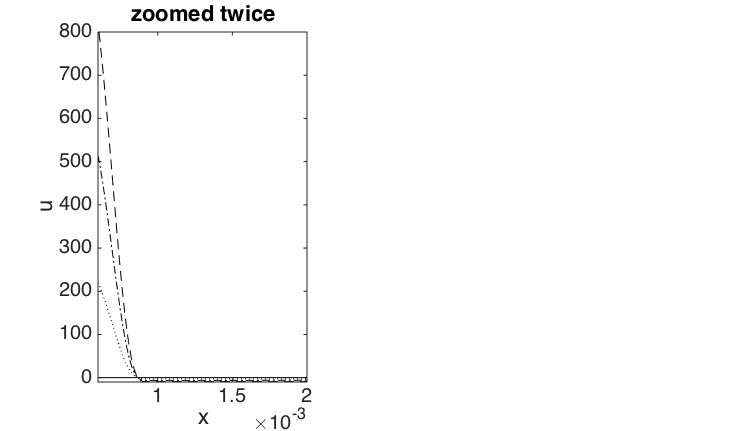} \hspace{-4.8 cm}
     \includegraphics[height=0.38\textwidth, width=0.51\textwidth]{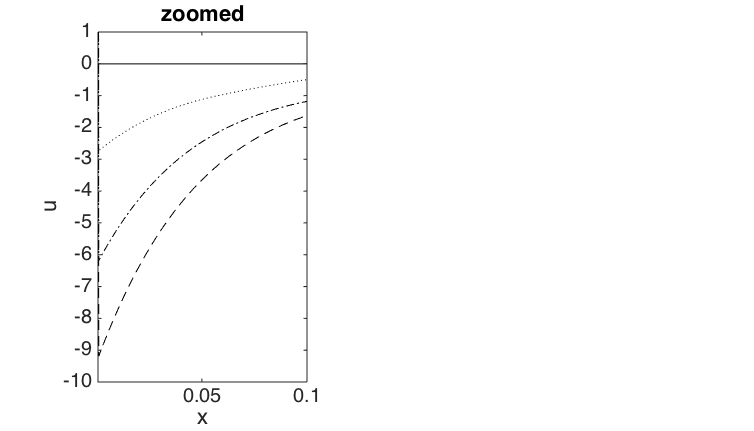}   \hspace{-4.8 cm}
        \includegraphics[height=0.38\textwidth, width=0.51\textwidth]{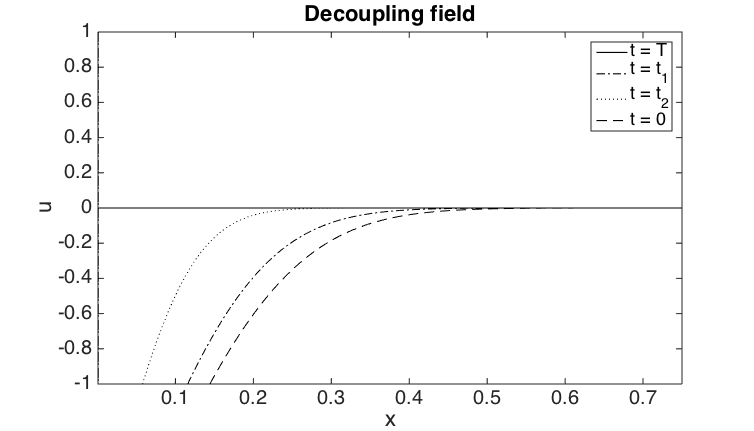}          
        \caption{Parameters $\alpha=1.5$, $\gamma = 0.0005$. Left and middle: Decoupling field $u(t,\cdot)$ for different $t$ and two ranges of (small) $x$. Right: Decoupling field $u(t,\cdot)$ for different $t$ and marcroscopic range.}
        \label{fig:decoupling}
\end{figure}

\subsubsection*{Analysis of optimal cost-loss pairs}

%{\color{red} \fbox{Text to be updated}}

Finally, we examine the pairs of costs and losses
that are obtained under the optimal policy and other heuristic strategies.

In Figure \ref{fig:front_diff_alpha}, 
%we vary $\gamma$ to show the optimal pairs $(C^\star_T(\gamma), L^\star_T(\gamma))$ as in \eqref{opt_loss} and \eqref{opt_cost} under different interaction strengths $\alpha$.
we vary $\gamma$ to trace out the curve $(C^\star_T(\gamma), L^\star_T(\gamma))$, where $C^\star_T(\gamma)$ and $L^\star_T(\gamma)$ are the costs and losses given by
\eqref{opt_loss} and \eqref{opt_cost} %, and derived from \eqref{p_PDE_beta} and \eqref{HJB},
 for the chosen $\gamma$.
 For a given cost, the graph gives the loss achievable under the optimal strategy.
 To achieve a smaller loss, a higher cost is generally incurred.

We focus first on the data for $\alpha=0.5$ and $\alpha=1$. In these cases, the uncontrolled system exhibits no jumps and cash injection simply reduces the losses.
For small $\gamma$, minimising the cost is the priority and the losses approach those of the uncontrolled system.
For growing $\gamma$, it becomes favourable to increase the cash injection and a significant reduction of losses can be achieved. This levels off for large $\gamma$
as the cap $b_{\max}$ on the cash injection rate limits the overall effect of bail-outs.

%\begin{figure}[t!]
%    \centering
%        \includegraphics[height=0.4\textwidth, width=0.7\textwidth]{Plots/cost_loss_pairs_alpha.eps}
%        \caption{$(C_T^\star(\gamma), L^\star_T(\gamma))$ for logarithmically spaced $\gamma \in [0.03,30]$, and different $\alpha$.
%        {\color{red} to do: 1) remove spurious points; 2) add the jumping loss function as plot?}}
%        \label{fig:mesh_conv}
%\end{figure}

\begin{figure}[t!]
    \centering
    \begin{subfigure}[t]{0.49\textwidth}
        %\centering
            \hspace{-1cm}
        \includegraphics[height=0.75\textwidth, width=1.15\textwidth]{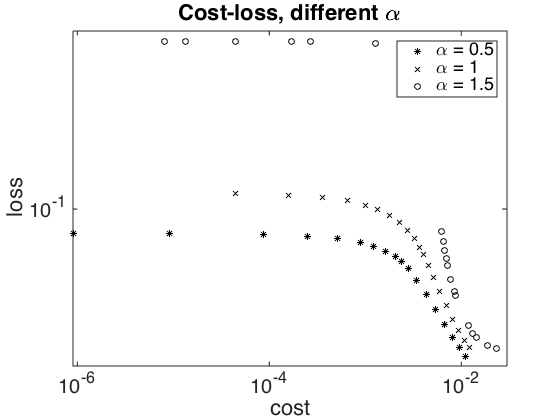}
        \caption{Pairs $(C^\star_T, L^\star_T)$ for different $\alpha$. }
        \label{fig:front_diff_alpha}
    \end{subfigure}%
    \hfill
    \begin{subfigure}[t]{0.49\textwidth}
        %\centering
        \hspace{-0.5cm}
        \includegraphics[height=0.75\textwidth, width=1.15\textwidth]{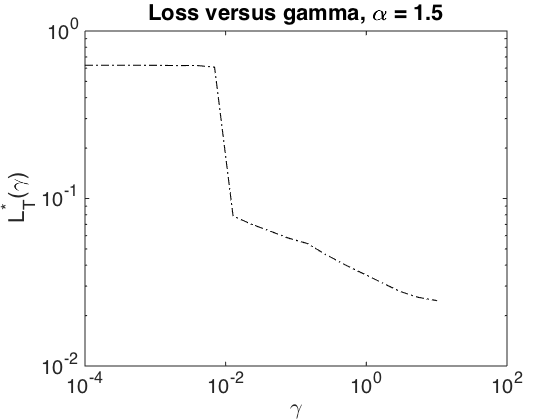}
        \caption{$L^\star_T$ as function of $\gamma$, $\alpha = 1.5$.}
       \label{fig:C_L_alpha_15}
    \end{subfigure}
    \caption{Cost $C^\star_T$ and loss $L^\star_T$ in the optimal regime for logarithmically spaced $\gamma \in [0.0001, 0.1]$ and different $\alpha$ in (a)
   and the dependence of the losses on $\gamma$ in (b).}
\end{figure}

For strong interaction, here exemplified by $\alpha=1.5$, we observe a discontinuity, which is further analysed in Figure \ref{fig:C_L_alpha_15}.
For $\gamma$ around 0.01, the optimal strategy switches from not preventing a jump to preventing a jump. This is manifested in Figure \ref{fig:C_L_alpha_15}
by an downward discontinuity  in the number of losses. %, and an upward discontinuity in the cost.
The optimal value of the central agent's control problem is also discontinuous in $\gamma$ at this point.
In other words, it is not possible to vary the capital injection to control the \emph{size} of the jump continuously. Rather, the possible jump size is restricted
by the constraint \eqref{physical} on physical solutions. Conversely, withdrawal of a small amount of cash by the central agent from a scenario with low losses can trigger
a large systemic event.

Note that Figure \ref{fig:C_L_alpha_15} also allows to deduce the relation between $\gamma$ and the threshold $\delta$ by looking for $\gamma$ such that $\gamma \in 
\operatorname{argmax}_{\gamma \in \mathbb{R}_+} C^\star_T(\gamma) + \gamma (L^\star_T(\gamma) - \delta)$, as explained in the introduction. Indeed, this corresponds to solving the outer optimization problem
$\max_{\gamma \in \mathbb{R}_+} g(\gamma)$ where $g(\gamma)= \min_{\beta} \mathcal{L}(\beta, \gamma)$ with $\mathcal{L}(\beta, \gamma)$ denoting the Lagrange function.
Under the assumption of no duality gap and a unique optimizer $\gamma$, $\gamma$ is necessarily determined via $L^\star_T(\gamma) =\delta$. As we observe a jump discontinuity of $\gamma \mapsto L^\star_T(\gamma)$, this suggests that there is a duality gap at least for certain values of $\delta$.

We proceed by comparing the costs and losses under the optimal strategy with some other heuristic strategies. 
As first benchmark, we consider a uniform strategy by which the central agent
injects cash at a constant rate $b_{\max}$ whenever an agent's value $X_t\le c$ for a constant $c$, which we vary, resulting in pairs $(C_T^{\mathrm{u}}(c), L_T^{\mathrm{u}}(c))$. 
%These are the curves shown as dashed in in Figure \ref{fig:compare_strats}.
The total cost here can be computed as $C_T^{\mathrm{u}}(c) = b_{\max} \cdot \int_0^T \int_0^c p^{\mathrm{u}}(t,x) \, dx dt$,
where $p^{\mathrm{u}}$ is the density of the %absorbed
regularised
process with such uniform (in time) control.

We also consider a `front loaded' strategy whereby  at the outset, for some chosen `floor' $d>0$, the central agent injects a lump sum of $d-X_{0-}$ into all players with 
$X_{0-}<d$, hence lifting their reserves up to $d$. Again, we vary $d$ to obtain a parametrised curve $(C_T^{\mathrm{f}}(d), L_T^{\mathrm{f}}(d))$.
%shown as dot-dashed in Figure \ref{fig:compare_strats}.
The total cost in this case is found as $C_T^{\mathrm{f}}(d) = \int_0^d (d-x) f(x) \, dx$.

The pairs of cost and loss are shown in Figure \ref{fig:compare_strats}. In particular, 
Figure \ref{fig:strats_a05} illustrates the case without jump for $\alpha=1$,
whereas in the situation of \ref{fig:strats_a15} with $\alpha=1.5$ there is a jump in the uncontrolled system, which can be avoided with sufficiently large control.
\begin{figure}[t!]
    \centering
    \begin{subfigure}[t]{0.49\textwidth}
        %\centering
            \hspace{-1cm}
        \includegraphics[height=0.77\textwidth, width=1.15\textwidth]{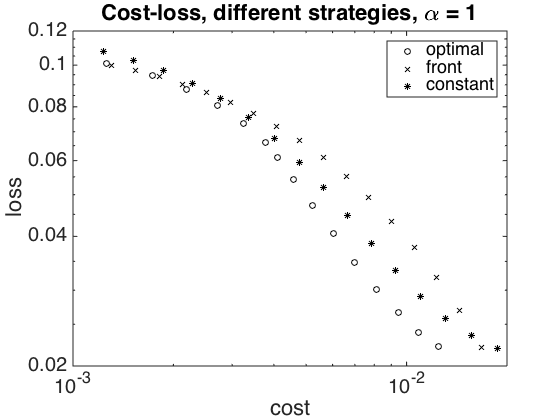}
        \caption{$\alpha=1$, no jumps.}
        \label{fig:strats_a05}
    \end{subfigure}
    \hfill
    \begin{subfigure}[t]{0.49\textwidth}
        %\centering
        \hspace{-0.5cm}
        \includegraphics[height=0.77\textwidth, width=1.15\textwidth]{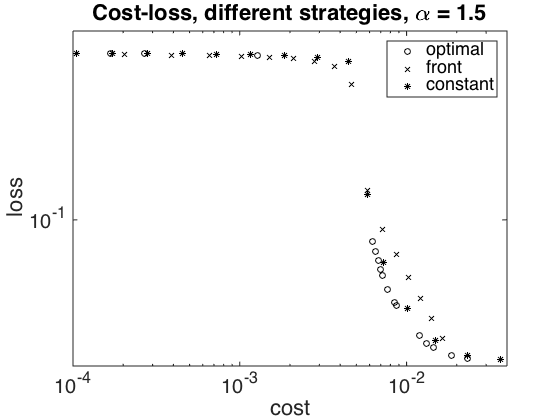}
        \caption{$\alpha=1.5$, jump possible.}
       \label{fig:strats_a15}
    \end{subfigure}
    \caption{Cost-loss pairs $(C^\star_T, L^\star_T)$ under optimal strategy compared to those for a constant strategy, $(C_T^{\mathrm{u}}, L_T^{\mathrm{u}})$, 
    and front-up strategy, $(C_T^{\mathrm{f}}, L_T^{\mathrm{f}})$, for
    two values of $\alpha$.}
    \label{fig:compare_strats}
\end{figure}
In both cases, the optimal strategy gives lower losses than the heuristic strategies for the same fixed cost.
Conversely, less cash injection is required for a given loss tolerance.\footnote{Note that the strategy with upfront payments is not in the class of Lipschitz feedback controls
for which the policy gradient method is designed.}
%The uniform strategy %without optimisation %higher cost does not necessarily lead to less loss.
%is often almost as effective as the optimal strategy, which is perhaps to be expected
%since the optimal strategy does not vary too much until close to expiry.
%The upfront strategy behaves reasonably for small costs, but is far from optimal if small losses are required.

We observe that we cannot enforce the \emph{sufficient} condition for avoiding jumps, i.e.,  $\alpha \| f\|_{\infty} < 1$, for any of these strategies.
It is clear that the strategy where the initial capital of all banks is raised to a certain minimum level $d$ satisfies $\mathbb{E}[X_{0-}]\ge d$ and hence the \emph{necessary} condition for avoiding jumps, $\mathbb{E}[X_{0-}] \ge \alpha/2$, holds for $d \geq \alpha/2$.
However, the sufficient condition can be violated even when all banks have a high initial capital. What would work to enforce the sufficient condition is to set
$X_{0-} \sim U(d, d+\alpha+\varepsilon)$ for some $\varepsilon >0$, $d \geq 0$ and $U$ the uniform distribution on $[d,d+\alpha+\varepsilon]$.
%\footnote{One can in principle envisage a procedure whereby the initial density is capped at a value  smaller than $1/\alpha$ and all mass beyond it re-distributed, but this would not seem an economically natural approach.{\color{orange}Stefan: wouldn't something like $X_{0-} = x +\frac{1}{\alpha} U$, where $U\sim U(0,1)$ is uniformly distributed and we let $x$ be arbitrarily large yield a more natural example?  }}

Considering the physical jump condition \eqref{physical}, for a jump to occur it matters how much of the surviving mass can be concentrated around zero at any given point. Intuitively, starting with a higher initial condition, the Brownian motion will diffuse the mass sufficiently and make large concentrations at zero less likely, hence preventing a jump.
Similarly, sufficiently large $\beta(t,x)$ for small $t$ and $x$ should transport mass away from zero and prevent a jump as long as the initial density satisfies $f(0+)<1/\alpha$ (which rules out an instantaneous jump).
Therefore, both the constant and optimal strategies should be able to prevent jumps for large enough $b_{\max}$.
A rigorous analysis, however, goes beyond the scope of this paper.

\appendix

\section{Proofs}\label{sec:proofs}

\subsection{Notation}\label{sec:app_notation}
Throughout the paper, $D([-1,\infty))$ denotes the space of \cadlag functions on $[-1,\infty)$ endowed with the $M_1$-topology, $C([0,\infty))$ denotes the space of
continuous functions on $[0,\infty)$ endowed with the topology of compact convergence, i.e., $f_n \to f$ in $C([0,\infty))$ if and only if $f_{n}\vert_{K} \to f\vert_{K}$ uniformly
for every compact $K \subseteq [0,\infty)$. If $S$ is a Polish space, we denote the space of probability measures on $S$ by $\mathcal{P}(S)$ and endow it with the topology of weak convergece,
i.e., we say that $\mu_n \to \mu$ in $\mathcal{P}(S)$ iff $\int_{S} F(x) \mathrm{d}\mu_n (x) \to \int_{S}F(x) \mathrm{d}\mu (x)$ for all $F \in C_b (S;\mathbb{R})$.
If $\mu \in \mathcal{P}(S)$ and $F \colon S \rightarrow \mathbb{R}$, we denote the integral of $F$ with respect to $\mu$ also with brackets, i.e., we write 
$\int_{S}F(x)~\mathrm{d}\mu(x) = \langle \mu, F \rangle$. Furthermore, if $\nu$ is the pushforward of the measure $\mu$ with respect to the map $T$, we denote this by $T(\mu) = \nu$.

\subsection{Existence of minimal solutions and optimizers} \label{app:existence}

\begin{lemma}\label{thm:crossingproperty}
For any $\beta \in \mathcal{B}_T$, define the process
\begin{equation}
Z_t = X_{0-}+\int_{0}^{t}\beta_s\,ds + B_t, \quad t \geq 0.
\end{equation}
Then, the process $Z$ satisfies the extended crossing property, i.e.,
\begin{equation}
\mathbb{P}\left(\inf_{0\leq s\leq h} (Z_{\tau+s} - Z_{\tau}) = 0 \right) = 0, \quad h > 0,
\end{equation}
for any stopping time $\tau$ with respect to $(\mathcal{F}_{t})_{t \geq 0}$, the natural filtration
generated by $Z$.
\end{lemma}
\begin{proof}
Let $\tau$ be a $(\mathcal{F}_t)_{t\geq 0}-$stopping time. Since $\beta \in \mathcal{B}_T$ is almost surely bounded, by Novikov's condition and Girsanov's theorem we may find an equivalent probability measure $\mathbb{Q}$ 
such that $Z$ is a $(\mathcal{F}_t)_{t \geq 0}$-Brownian motion under $\mathbb{Q}$. The strong Markov property of Brownian motion then yields
\begin{equation*}
\mathbb{Q}\left(\inf_{0\leq s\leq h} (Z_{\tau+s} - Z_{\tau}) = 0 \right) = \mathbb{Q}\left(\inf_{0\leq s\leq h} Z_s = 0 \right) = 0, \quad h > 0
\end{equation*}
and by the equivalence of $\mathbb{P}$ and $\mathbb{Q}$ the claim follows.
\end{proof}

\subsection{Existence of solutions}
\begin{lemma}\label{thm:integralsuniform}
Suppose that $f_n \to f$ in $S_T$. Then, $\int_{}^{\cdot} f_n(s)\,ds \to \int_{}^{\cdot} f(s)\,ds$ uniformly in $t$ 
on any compact subset of $[0,\infty)$.
\end{lemma}
\begin{proof}
Let $T' > 0$. By weak $L^2([0,\infty))$-convergence, we have $\int_{0}^{t} f_n(s)\,ds \to \int_{0}^{t} f(s)\,ds$ for any $t \in [0,T']$.
Let $\epsilon > 0$ and choose $t_k$ with $0 = t_0 \leq t_1 \leq \dots \leq t_m = T'$ such that 
\begin{equation}\int_{t_i}^{t_{i+1}} f(s)\,ds < \epsilon/2, \quad i \in \{0,\dots,m-1\}.
\end{equation}
Choose $n$ large enough such that $\left|\int_{0}^{t_i} f_n(s)-f(s)\,ds\right| < \epsilon/2$ for all $i \in \{0,\dots,m\}$. We obtain, for $t \in [t_i,t_{i+1}]$, 
\begin{align*}
&\int_{0}^{t} f_n(s)\,ds - \int_{0}^{t} f(s)\,ds \leq \int_{0}^{t_{i+1}} f_n(s)\,ds - \int_{0}^{t_{i+1}} f(s)\,ds + \epsilon/2 \leq \epsilon, \\
&\int_{0}^{t} f_n(s)\,ds - \int_{0}^{t} f(s)\,ds \geq \int_{0}^{t_{i}} f_n(s)\,ds - \int_{0}^{t_{i}} f(s)\,ds - \epsilon/2 \geq -\epsilon, 
\end{align*}
which yields the claim.
\end{proof}

\begin{proof}[Proof of Theorem \ref{thm:lambdaconvergence}]. Step 1: We construct the reference probability space $\mathscr{S}$. 

Since the sequence $\law((X_{0-}^n,B^n))$ is constant and the space $[0,\infty) \times C([0,\infty))$ (endowed with the product topology of Euclidean and uniform topology) is Polish, the sequence $(X_{0-}^n,B^n)$ is tight on $[0,\infty) \times C([0,\infty))$. Since $S_T$ is compact, the sequence $(X_{0-}^n,B^n,\beta^n)$ is tight on $[0,\infty) \times C([0,\infty))\times S_T$. By Prokhorov's theorem, after passing to a subsequence if necessary, we may assume that $\law((X_{0-}^n,B^n,\beta^n)) \to \law((X_{0-},B,\beta))$ in $\mathcal{P}([0,\infty) \times C([0,\infty)) \times S_T)$.
By the Skorokhod representation theorem, we may without loss of generality assume that $(X_{0-}^n,B^n,\beta^n) \to (X_{0-},B,\beta)$ holds almost surely on some probability space, which we denote by $\mathscr{S} = (\Omega, \mathcal{F}, \mathbb{P})$. Note that this is 
possible since the property of $\Lambda$ solving \eqref{eq::controlled}-\eqref{eq::controlled2} only 
depends on the (joint) law of $(X_{0-},B,\beta)$. We then define $\mathcal{F}_t$ to be $\sigma(\{(X_{0-},B_s,\int_{0}^{s} \beta_u \,ds), s \leq t \})$ for $t \geq 0$. Since $\beta$ is measurable and $(\mathcal{F}_t)_{t \geq 0}$-adapted, it admits a progressively measurable modification, which we again denote by $\beta$, and we see that $\beta \in \mathcal{B}_T(\mathscr{S})$. We next show that $B$ is a $(\mathcal{F}_t)_{t \geq 0}$-Brownian motion: by the continuous mapping theorem and the continuity of the coordinate projections, it follows that $\law(B)$ is the Wiener measure. Let $0 < r_1 < \dots < r_k < s < t$ and let $f_i \in C_b(\mathbb{R}^2;\mathbb{R}), i=1,\dots,k$ and $ g \in C_b(\mathbb{R};\mathbb{R})$ and set $F(w^1,w^2) := \prod_{i=1}^{k} f_i (w_{r_i}^1,w_{r_i}^2)$, then we have by dominated convergence
\begin{align*}
\mathbb{E} \left[F\left(B,\int_{0}^\cdot \beta_u \,du\right)g(B_t - B_s) \right] &= 
\lim_{n\to\infty}\mathbb{E} \left[F\left(B^n,\int_{0}^\cdot \beta_u^n\,du\right)g(B_t^n - B_s^n) \right] \\
&= \lim_{n\to\infty}\mathbb{E} \left[F\left(B^n,\int_{0}^\cdot \beta_u^n\,du\right)\right] \mathbb{E} \left[g(B_t^n - B_s^n) \right] \\
 &= \mathbb{E}\left[F\left(B,\int_{0}^\cdot \beta_u\,du\right)\right] \mathbb{E} \left[g(B_t - B_s) \right].
\end{align*}
Since the Borel $\sigma$-algebra on $C([0,\infty))$ is generated by the evaluation mappings, we obtain that $B_t - B_s$ is independent of $\mathcal{F}_s$ for $0 < s < t$, and therefore $B$ is an $(\mathcal{F}_t)_{t \geq 0}$-Brownian motion. We have shown that $\mathscr{S}$ is an admissible reference space.\\

Step 2: Since $M$ is compact, after passing to
subsequences if necessary, we may assume that $\frac{1}{\alpha} \Lambda^n \to \frac{1}{\alpha}\Lambda$ in $M$. We now show that $(X_{0-},B,\beta,\Lambda)$ solves \eqref{eq::controlled}-\eqref{eq::controlled2} on $\mathscr{S}$. Let $\ExtendedE = C([0,\infty)) \times M$ be endowed with the product topology and define $Z_t^n:=X_{0-}^n + B_t^n + \int_{0}^{t} \beta_s^{n}\,ds $ and $Z$ analogously as $Z_t = X_{0-} + B_t + \int_{0}^{t} \beta_s\,ds$. By Lemma \ref{thm:integralsuniform}, it follows that $Z^n\to Z$ in $C([0,\infty))$ almost surely. Since $\Lambda$ is deterministic, it follows that
$\xi^n := (Z^n,\frac{1}{\alpha}\Lambda^n) \to (Z,\frac{1}{\alpha}\Lambda) =: \xi$ in distribution on $\ExtendedE$. We introduce some notation: define $\iota \colon \ExtendedE \rightarrow D([-1,\infty))$ for $w \in C([0,\infty))$ and $\ell \in M$ as 
\begin{align}\label{eq::iotadef}
\iota(w,\ell)_t := \begin{cases} 
w_0, \quad & t \in [-1,0), \\
w_t - \alpha\ell_t, \quad & t \in [0,\infty).
\end{cases}
\end{align}
For $t \in \R$ and $x\in D([-1,\infty))$, define the path functionals
$\tau_0 (x) := \inf\{s \geq 0: x_s \leq 0\}$ and $\ellfunc_t (x) := \ind{\{\tau_0 (x) \leq t\}}$. Then, $(X_{0-},B,\beta,\Lambda)$ is a solution to \eqref{eq::controlled}-\eqref{eq::controlled2} on $\mathscr{S}$ if and only if $\alpha\mathbb{P}(\tau_0 (Z - \Lambda) \leq t) = \Lambda_t$. We may write this condition equivalently on the canonical
path space $D([-1,\infty))$ as
$\alpha\langle\iota(\xi),\ellfunc_t \rangle = \Lambda_t$ for $t \geq 0$. Note that with the notational conventions explained in Section \ref{sec:app_notation}, $\iota(\xi)$ denotes the pushforward of the measure $\xi$ by the map
$\iota$ and $\langle\iota(\xi),\ellfunc_t \rangle$ denotes the integral of the functional $\ellfunc_t$ with respect to $\iota(\xi)$. Since $Z$ satisfies the extended crossing property by Lemma \ref{thm:crossingproperty}, $\iota(\xi)$ satisfies the crossing property (cf the proof of Lemma 5.5 in \cite{CRS}). Lemma 5.3 in \cite{CRS} and Step 1 imply that
\begin{equation}
\Lambda = \lim_{n\to\infty}\Lambda^n = \lim_{n\to\infty} \alpha\langle\iota(\xi^n),\lambda \rangle = \alpha\langle\iota(\xi),\lambda \rangle. 
\end{equation}
in $M$.
\end{proof}

\subsection{Proof of Theorem \ref{thm:finitedimconvergence}} \label{app::finitedimconvergence}

\begin{proof}
The proof is to a large extent analogous to the proof of Proposition 5.6 in \cite{CRS}. Define
\begin{equation*}
\EmpMext_N := \frac{1}{N} \sum_{i=1}^{N} \delta_{(X_{0-}^{i,N} + B^{i,N},\beta^{i,N},L^N)}.
\end{equation*}
Since $\EmpMext_N$ is a random probability measure on $C([0,\infty)) \times S_T \times M$ and the spaces $S_T$ and $M$ are compact, 
$\EmpMext_N$ is tight by the same reasoning as in Corollary 4.5 in \cite{CRS}. Therefore, after passing to a subsequence if necessary,
we may assume that $\law(\EmpMext_N) \to \law(\EmpMext)$ for some random probability measure $\EmpMext$ on $C([0,\infty)) \times S_T \times M.$
By Skorokhod representation, we may assume without loss of generality that the convergence happens almost surely on the same probability space $\mathscr{S}$.
Arguing in the same fashion as in the proof of Lemma 5.4 in \cite{CRS}, we see that for almost every $\omega \in \Omega$, if $\law((W,\beta,\mathsf{L})) = \EmpMext(\omega)$, 
then $W-W_{0}$ is a Brownian motion with respect to the filtration generated by $(W,\int_{0}^{\cdot}\beta_s\,ds,\mathsf{L})$. For $(w,b,\ell) \in C([0,\infty)) \times S_T \times M$, set
$\embedding(w,b,\ell)_t := \iota(w,\ell)_t + \int_{0}^{t} b_s\,ds$, where $\iota$ is defined as in \eqref{eq::iotadef}. By Lemma \ref{thm:integralsuniform}, the map $b \mapsto \int_{0}^{\cdot} b_s\,ds$ is continuous from $S_T$ to $C([0,\infty))$, and therefore also as a map from $S_T$ to $D([-1,\infty))$. Theorem 4.2 in \cite{CRS} together with Corollary 12.7.4 in \cite{Whi} then shows that $\embedding$ is continuous. Since $\embedding(\EmpMext_N) = \EmpM_N$, the continuous mapping theorem implies that $\embedding(\EmpMext) = \mu$. Applying the continuous mapping theorem to $(w,b,\ell) \mapsto w_0$, we see that $\law(W_0) = \nu_0$ holds almost surely. It remains to check that if $\law(W,\beta,\mathsf{L}) = \EmpMext(\omega)$, then $\mathsf{L}_t \equiv \langle \mu(\omega),\ellfunc_t \rangle$ for $t \geq 0$ holds $\EmpMext(\omega)$-almost surely for almost every $\omega \in \Omega$. This can be checked as in Step 1 of the proof of Proposition 5.6 in \cite{CRS}, making use of Lemma \ref{thm:crossingproperty} to show that $\mu(\omega)$ satisfies the crossing property (almost surely).
\end{proof}

\section{Further numerical details and tests}

We here report further details and tests of our numerical procedure.

\subsection{Smoothing}
\label{subsec:smooth}

Let us start by precisely specifying the smoothing functions $\phi^h$ and $\Phi^h$ used in the regularization procedure of the objective function and the dynamics. 
We choose  $\phi^h(x) = \phi(x)/h$, where,
\begin{eqnarray}
\phi(x) = \left\{\begin{array}{rl}
\frac{1}{\mathcal{I}} \exp(-1/(x (1-x)), & x \in [0,1], \\
0, & \text{else}.
\end{array}
\right.
\end{eqnarray}
where $\mathcal{I} = 0.007029858406609$ normalises the integral (close) to 1.
The function $\phi^h$ and its first two derivatives are shown in Figure \ref{fig:phi_der}
for $h=10^{-3}$. %, which was used in many of the numerical examples.
Note in particular the large positive and negative values of $\partial^2 \phi^h$.

\begin{figure}[t!]
%    \centering
    \begin{subfigure}[t]{0.32\textwidth}
        %\centering
            \hspace{-0.9cm}
        \includegraphics[height=0.75\textwidth, width=1.15\textwidth]{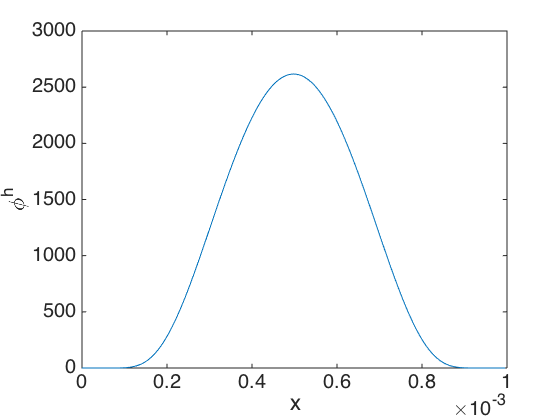}
        \caption{$\phi^h$}
        \label{fig:phi}
    \end{subfigure}
    \hfill
    \begin{subfigure}[t]{0.32\textwidth}
        %\entering
        \hspace{-0.9cm}
        \includegraphics[height=0.75\textwidth, width=1.15\textwidth]{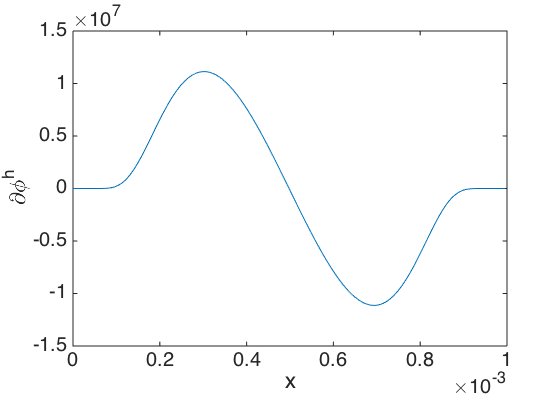}
        \caption{$\partial \phi^h$}
       \label{fig:d_phi}
    \end{subfigure}
\hfill
    \begin{subfigure}[t]{0.32\textwidth}
        %\centering
        \hspace{-0.9cm}
        \includegraphics[height=0.75\textwidth, width=1.15\textwidth]{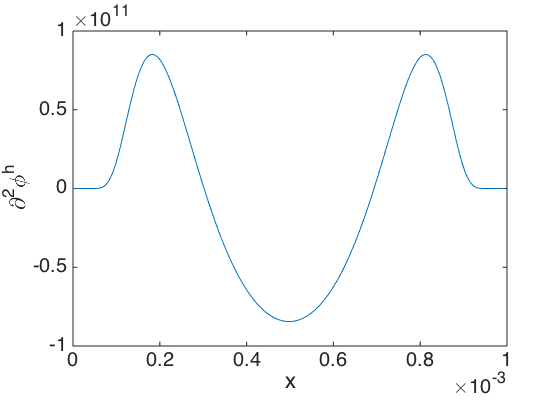}
        \caption{$\partial^2 \phi^h$}
       \label{fig:d2_phi}
    \end{subfigure}     
    \caption{Smoothed Dirac delta and its derivatives, for $h=10^{-3}$.}
    \label{fig:phi_der}
\end{figure}

We also choose for some $\kappa>0$, $\Phi^h(x) = \Phi(x/(\kappa h))$, where $\Phi(x) = \mathcal{N}((x+1/2)/(x*(1+x)))$,
$x\in (-1,0)$, and 0 else,
with $\mathcal{N}$ the standard normal CDF.

\subsection{Mesh convergence}
\label{subsec:meshconv}

Here, we demonstrate the convergence of the finite difference approximation in two more
settings. In Table \ref{table:mesh_convergence_cost} for the cost, for $\alpha = 1.5$, and in Table 
\ref{table:mesh_convergence_cost} again for the losses, but now for $\alpha = 0.5$. The notation with $\theta_N$
$\rho_N$, $\vartheta_N$ and $\varrho_h$ is analogous as before.

\begin{table}[ht!]
%\centering
\hspace{-0.6 cm}
\begin{tabular}{|r|r|c|c|c|c|c|c|c|c|}
\hline
&& & $10^3 \cdot h$ & $1$ & $2^{-1}$ & $2^{-2} $ & $2^{-3}$ & $2^{-4} $ & $2^{-5} $ \\ \hline \hline
$\frac{N}{10^2}$ & $N_x/10^3$ & $10^3 \cdot \theta_N$ & $\rho_N$ &&&&&&\\  \hline  
$1$ & $3.75$ &-1.650 & 1.86 &  1.2548  &  1.1449  &  1.0225  &  0.8010  &  0.0526  &  0.5638  \\
$2$ & $7.5$ &-0.887 & 0.83 & 1.2383  &  1.1115  &  1.0085  &  0.9004  &  0.6965  &  0.0343 \\
$4$ & $15 $ &-1.066 & 1.99 &  1.2294  &  1.0970  &  0.9939  &  0.9246  &  0.6131  &  0.4506  \\
$8$ & $30$  &-0.534 & 2.00 &  1.2187  &  1.0846  &  0.9839  &  0.9126  &  0.8484  &  1.0968  \\
$16 $ & $60$ &-0.267 & --- & 1.2134  &  1.0784  &  0.9790  &  0.9126  &  0.8644  &  0.8199  \\
$32 $ & $120$ &--- & --- &  1.2107  &  1.0753  &  0.9764  &  0.9123  &  0.8701  &  0.8365  \\  \hline
&& &$10^3 \cdot \vartheta_h$ & -1.354  & -0.989 &  -0.641 &  -0.422  & -0.335 & ---  \\
\hline
&& &$\varrho_h$ & 1.37   & 1.54  &  1.51 &   1.25 & --- & ---  \\ \hline
%&CPU (sec) && 0.44 & 1.2 & 4.2 & 17 & 83 & 427 \\ \hline
\end{tabular}
\caption{Mesh convergence, 100 $\cdot$ costs, $\alpha=1.5$ and $\gamma=0.1$.
}
\label{table:mesh_convergence_cost}
\end{table}

\begin{table}[ht!]
%\centering
\hspace{-0.6 cm}
\begin{tabular}{|r|r|c|c|c|c|c|c|c|c|}
\hline
&& & $10^3 \cdot h$ & $1$ & $2^{-1}$ & $2^{-2} $ & $2^{-3}$ & $2^{-4} $ & $2^{-5} $ \\ \hline \hline
$\frac{N}{10^2}$ & $N_x/10^3$ & $10^3 \cdot \theta_N$ & $\rho_N$ &&&&&& \\  \hline  
$1$ & $3.75$ &3.802 & 1.99 &  7.3549  &  7.4198  &  7.6336  &  8.7134  &  735.42   &      0  \\
$2$ & $7.5$ &1.910 & 2.03 & 7.3929  &  7.4573  &  7.4919  &  7.6925  &  8.7711  & 735.91 \\
$4$ & $15 $ &0.941 & 2.00 & 7.4120  &  7.4771  &  7.5103  &  7.4765  &  9.5999  & -4.196 \\
$8$ & $30$  &0.470 & 2.00 & 7.4214  &  7.4868  &  7.5203  &  7.5373  &  7.5867  &  6.2291  \\
$16 $ & $60$ &0.235 & --- & 7.4261  &  7.4916  &  7.5251  &  7.5421  &  7.5504  &  7.5550  \\
$32 $ & $120$ &--- & --- &  7.4285  &  7.4939  &  7.5275  &  7.5445  &  7.5531  &  7.5571  \\  \hline
&& &$10^3 \cdot \vartheta_h$ & 0.6547  &  0.3353  &  0.1701  &  0.0860  &  0.0398 & ---  \\
\hline
&& &$\varrho_h$ & 1.95  &  1.97 &   1.97  &  2.15 & --- & ---  \\ \hline
%&CPU (sec) && 0.44 & 1.2 & 4.2 & 17 & 83 & 427 \\ \hline
\end{tabular}
\caption{Mesh convergence, 100 $\cdot$ losses, $\alpha=0.5$ and $\gamma=0.1$.
}
\label{table:mesh_convergence_loss}
\end{table}

In these settings, the behaviour in $h$ is somewhat better than in Table \ref{table:mesh_convergence}
(losses for $\alpha = 1.5$, i.e.\ with jump in the uncontrolled case), but as there,
a good approximation is only achieved if the mesh size is sufficiently small in comparison with $h$.
%{\color{red} \fbox{Here, in particular for the costs it looks like that for fixed large $N$ it also converges in $h$.}}

\subsection{Gradient iteration}

Finally we conducted further tests in view of the mesh refinement and the role of the step size $\tau$ in the gradient iteration. 

Figure \ref{subfig:PGM_N}, left, illustrates that the convergence is robust with respect to mesh refinement, i.e., the number of iterations required
for a prescribed accuracy does not increase significantly as the number of time steps and mesh points increases simultaneously.

In Figure \ref{subfig:PGM_tau} we investigate the effect of the step size $\tau$. Choosing $\tau$ small leads to poor convergence, while $\tau=0.3$ is optimal among the
values presented here. Picking even larger step sizes can lead to divergence.

\begin{figure}[t!]
    \centering
        \begin{subfigure}[t]{0.49\textwidth} 
        \includegraphics[height=0.78\textwidth, width=0.98\textwidth]{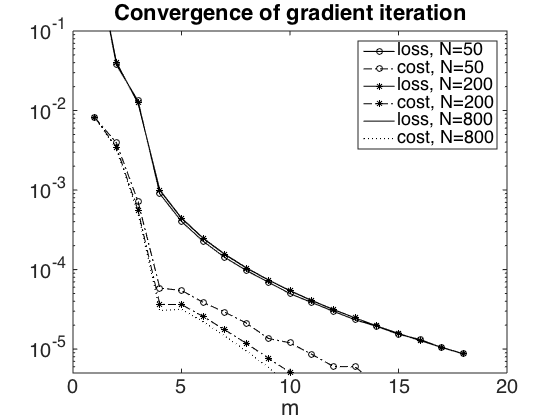}
        \caption{$\alpha=0.5$, $\gamma=1$, varying $N$}
        \label{subfig:PGM_N}
        \end{subfigure}
        \hfill
        \begin{subfigure}[t]{0.49\textwidth} 
              \includegraphics[height=0.78\textwidth, width=0.98\textwidth]{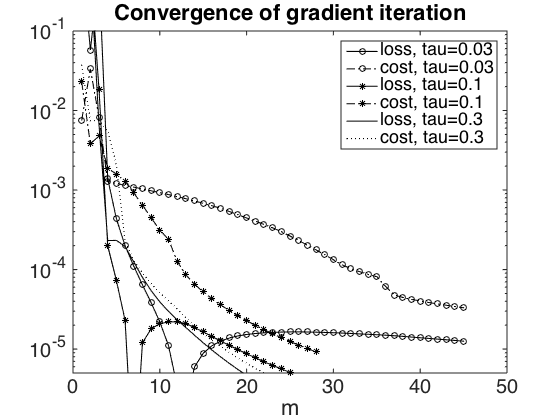}
              \caption{$\alpha=1.5$, $\gamma=1$, $N=800$, varying $\tau$}
              \label{subfig:PGM_tau}
            \end{subfigure}  
        \caption{Convergence of $C$ and $L$ over policy gradient iterations.}
        \label{fig:PGM_conv2}
\end{figure}

\section*{Declaration}

{\bf Funding:} %\newline
The authors gratefully acknowledge financial support by the Vienna Science and Technology Fund (WWTF) under grant MA16-021 and  by the Austrian Science Fund (FWF) through grant Y 1235 of the START-program.

%%%%%%%%%%%%%%%%%%%%%%%%%%%
\bibliography{bibliography}
\bibliographystyle{amsplain}

%%%%%%%%%%%%%%%%%%%%%%%%%%

%\bigskip\bigskip

\end{document}